\documentclass[11pt]{amsart}
\usepackage{graphicx,amssymb,amsmath,amsthm}
\usepackage{comment,cite,color}
\usepackage{cite,color}
\usepackage{algorithmic}
\usepackage{mathrsfs}
\usepackage[ruled]{algorithm}
\usepackage{epsfig}

\newtheorem{thm}{Theorem}
\newtheorem{cor}[thm]{Corollary}
\newtheorem{lem}[thm]{Lemma}

\theoremstyle{definition}

\newtheorem{remark}{Remark}

\newcommand{\R}{\mathbb{R}}

\newcommand{\Z}{\mathbb{Z}}
\newcommand{\N}{\mathbb{N}}

\newcommand{\innerp}[1]{\langle {#1} \rangle}
\newcommand{\biginnerp}[1]{\left\langle {#1} \right\rangle}
\newcommand{\sinc }{{\rm sinc}}
\newcommand{\abs}[1]{\lvert#1\rvert}
\newcommand{\bigabs}[1]{\big\lvert#1\big\rvert}
\newcommand{\Bigabs}[1]{\Big\lvert#1\Big\rvert}
\newcommand{\Dev}[3]{\frac{d^{#2}}{d#1^{#2}}#3}
\date{}

\begin{document}
\bibliographystyle{plain}
\title[On  B-spline framelets derived from UEP]{On  B-spline framelets derived from the unitary extension principle}

\author{Zuowei Shen}
\thanks{ Zuowei Shen was supported by R146-000-113-112 from the National University of Singapore.
Zhiqiang Xu was supported  by NSFC grant 11171336 and by the Funds for Creative
Research Groups of China (Grant No. 11021101). }
\address{Department of Mathematics, National University of Singapore, Block S17, 10
Lower Kent Ridge Road, Singapore, 119076}
\email{matzuows@nus.edu.sg}

\author{Zhiqiang Xu}

\address{LSEC, Inst.~Comp.~Math., Academy of Mathematics and System Sciences,  Chinese Academy of Sciences, Beijing, 100091, China}
\email{xuzq@lsec.cc.ac.cn}

\maketitle

\begin{abstract}

Spline  wavelet tight frames of \cite{RonShen97}  have been used widely in frame based image analysis and restorations  (see, e.g. survey articles  \cite{DongShen10, Shen10}). However, except for the tight  frame property and the approximation order of the truncated series  (see \cite{RonShen97, DaubechiesHanRonShen}), there are few other properties of this family of  spline  wavelet tight frames to be known.  This  paper is to  present  a few  new properties of this family  that  will provide further understanding of it and, hopefully, give some indications why it is efficient in image analysis and restorations. In particular, we present a recurrence formula of computing  generators of higher order spline wavelet tight  frames from the lower order ones.  We also represent each generator of spline wavelet tight frames as certain order of derivative of some univariate box spline. With this,  we further show that  each  generator of sufficiently high order spline wavelet tight frames  is close to  a right order of derivative of a properly scaled Gaussian  function.   This leads to  the result that the wavelet system generated by a finitely many consecutive derivatives   of a  properly scaled  Gaussian function forms a frame whose frame bounds can be  almost tight.
\end{abstract}

\section{Introduction}

 This paper is to  investigate  the family of the  spline wavelet  tight frames derived from  \cite{RonShen97}.
 We start with basic notions.
For given $\Psi:=\{\psi_1,\ldots,\psi_r\}\subset L_2(\R)$, the {\em wavelet system}  generated by $\Psi$ is defined as
$$
X(\Psi):=\{\psi_{\ell,n,k}:=2^{n/2}\psi_\ell(2^n\cdot-k): 1\leq \ell \leq r;\,\, n,k\in \Z\}.
$$
The system $X(\Psi)\subset L_2(\R)$ is called a {\em  tight frame} if
 $$
 f=\sum_{g\in X(\Psi)}\innerp{f,g}g
 $$
 holds for all $f\in L_2(\R)$.
 If $X(\Psi)\subset L_2(\R)$ is a  tight frame system of $L_2(\R)$ generated by a Multiresolution analysis (MRA), then  its  generators $\Psi$ are called as a {\em framelet}.

 The MRA starts from a refinable function $\varphi$. A
compactly supported function $\varphi$ is refinable if it satisfies a
refinement equation
\begin{equation} \label{refinable}
\varphi(x) = 2 \sum_{j \in \Z} a_j \varphi(2 x-j),
\end{equation}
for some sequence $a\in\ell_2(\Z)$. Refinable equation \eqref{refinable} can be written via its Fourier transform
as
\[\widehat{\varphi}(\omega)=\widehat a(\frac{\omega}{2})\cdot \widehat{\varphi}(\frac{\omega}{2}),\quad\text{a.e.}\quad\omega\in\R.\]
We call the sequence $a$ the refinement mask of $\varphi$ and
$\widehat{a}(\omega)$ the refinement symbol  of $\varphi$. Here,
 we use $\widehat{f}$ to denote the Fourier transform of $f\in L_1(\R)$, which is defined as
$$
\widehat{f}(\omega):=\int_{-\infty}^\infty f(x)\exp(-i\omega x)dx.
$$

 For a  refinable function $\varphi\in L_2(\R),$
let $V_0$ be the closed shift invariant space generated by
$\{\varphi(\cdot-k): k\in\Z\}$ and $V_j:=\{f(2^j\cdot): f\in V_0$\},
$j\in\Z$. It is  known that when $\varphi$ is compactly supported,
then $\{V_j\}_{j\in\Z}$ forms a multiresolution analysis. Recall
that a multiresolution analysis is a family of closed subspaces
$\{V_j\}_{j\in\Z}$ of $L_2(\R)$ that satisfies: (i) $V_j \subset
V_{j+1},$ (ii) $\bigcup_j V_j$ is dense in $L_2(\R)$ and (iii)
$\bigcap_j V_j=\{0\}$ (see \cite{carl:ca:93} and \cite{JS}).

A special family of refinable functions is B-splines.
 Let $\varphi^{(m)}$ be the centered B-spline of order $m$, which is defined in Fourier domain by
\begin{equation}\label{bspline}
\hat{\varphi}^{(m)}(\omega)=e^{-\frac{i\omega j_m}{2}}\sinc(\frac{\omega}{2})^m,
\end{equation}
where
\begin{equation}\label{defjm}
\sinc(x):=\begin{cases}
\sin(x)/x, & \hbox{for $x\neq 0$}\\
1,  & \hbox{for $x=0$}
\end{cases};
\quad\hbox{and}\quad
 j_m:=\begin{cases}
0, & \hbox{ $m$ is even}\\
1,  & \hbox{ $m$ is odd}
\end{cases} .
\end{equation}
Then $\varphi^{(m)}$ is refinable with refinement symbol
$$
\hat{a}^{(m)}(\omega)=e^{-\frac{i\omega j_m}{2}}\cos^m(\frac{\omega}{2}).
$$

  The tight framelets can be constructed by the unitary extension principle (UEP) of \cite{RonShen97} from a given  multiresolution
analysis. For a given B-spline $\varphi^{(m)}$  of order $m$, it was
 shown in \cite{RonShen97} that the $m$ functions, $\Psi^{(m)}=\{{\psi}_\ell^{(m)}: \ell=1,\ldots,m\}$, defined in Fourier domain by
\begin{equation}\label{eq:pside}
\hat{\psi}_\ell^{(m)}(\omega):=i^\ell e^{-\frac{i\omega j_m}{2}}\sqrt{{m\choose \ell}}\frac{\cos^{m-\ell}(\omega/4)\sin^{m+\ell}(\omega/4)}{(\omega/4)^m},
\end{equation}
 form a tight wavelet frame in $L_2(\R)$, i.e. $\Psi^{(m)}$ is a framelet set.
  We call $\Psi^{(m)}$ as the {\em B-spline framelet of order $m$}.
  The B-spline framelet is  either symmetric or anti-symmetric and has small supports for a given smoothness order.
 Similar with B-splines, each B-spline framelet  has an  analytic form.

Since the publication of the unitary extension principle (UEP)  of \cite{RonShen97} in 1997,   there are many theoretic developments  and applications of MRA based wavelet frames. In particular, the B-spline framelets $\Psi^{(m)}$ derived from the UEP in \cite{RonShen97} are widely used in various applications, which include image inpainting in \cite{CCS:ACHA:08};  image denoising in \cite{COS:XX:09}; high and super resolution  image reconstruction in \cite{CRSS:ACHA:04};  deblurring and blind debluring in \cite{COS:XX:08, COS:XX:09, CJLS, CS:NM:07}; and image segmentation in \cite{DCS}. The interested reader should consult the survey articles \cite{DongShen10,Shen10} for details.

The paper is organized as follows:  Section 2, which contains two subsections,  is for some basic properties of B-spline framelets.
In particular, in sub-section 2.1, we present recurrence formulas for B-spline framelets $\psi^{(m)}_\ell$, in which the well-known
recurrence formula of B-splines can be
viewed as a special form of recurrence formulas of B-spline framelets.
 This gives a fast algorithm for computing them.    We   further show  that the
 B-spline framelets can be derived from the $\ell$th derivative of
  some univariate  box spline in sub-section 2.2, This was implicity used in \cite{CDOS11}, to
  approximate the some derivatives of a function.
In Section 3, we investigate the asymptotic property of B-spline framlets $\psi_\ell^{(m)}$, $\ell=1,\ldots,m$. We firstly prove that the
univariate box splines  defined in Section 2 uniformly converge to a Gaussian function under a mild condition, and we further  show that
$$
\max_{1\leq\ell \leq m}\max_{x\in \R}\abs{\psi_{\ell}^{(m)}(x)-G_\ell^{(m)}(x)}\lesssim \frac{(\ln m)^{5/2}}{m^{3/2}},
$$
where $G_\ell^{(m)}$ is the $\ell$th derivative of some scaled Gaussian function $G(x)$ (see Section 3.2 for the detailed definition).

This leads to discover that wavelet system generated by a finite number of consecutive directives of scaled Gaussian function form a frame whose bounds are almost tight, and that is done in Section 4.

\section{Properties of B-spline framelets }
In this section, we give a  recurrence  formula of the B-spline framelets  which computes
higher order  framelets from lower order ones. We also show that one can
 represent the derivatives of higher order
framelets by lower order framelets.  Furthermore, we
derive  another  set of formulas that represents each framelet as a derivative of a univariate box spline which is already implicitly used
in \cite{CDOS11}, where a theory is developed to connect the PDE based  and spline wavelet based image restorations.

\subsection{Recurrence  formulas of  B-spline framelets}

 While the recurrence formulas for B-splines and their  derivatives are well-known (see \cite{deboor01}), the corresponding formulas for
  B-spline framelets are not available yet. This section is to establish such formulas.
 Let $B_m:=\varphi^{(m)}(\cdot+j_m/2)$, where $\varphi^{(m)}$ is given in \eqref{bspline} and $j_m$ is defined \eqref{defjm}.
 Recall the following  well-known recurrence formula of B-splines:
\begin{equation}\label{eq:Bsplinere}
B_{m+1}(x)= \frac{2x+m+1}{2m}B_m\left(x+\frac{1}{2}\right)+\frac{m+1-2x}{2m}B_m\left(x-\frac{1}{2}\right).
\end{equation}
Based on (\ref{eq:Bsplinere}), one can  compute B-splines fast and easily which  makes
B-splines useful.  The derivative of  B-splines can be computed by the lower order
splines as given below:
\begin{equation}\label{eq:Bsplinedev}
\Dev{x}{}{B_{m+1}{(x)}}=B_{m}\left(x+\frac{1}{2}\right)-B_{m}\left(x-\frac{1}{2}\right).
\end{equation}
The aim of this section  is to give  the corresponding formulas  for the B-spline framelets $\psi_\ell^{(m)},\ell=1,\ldots,m$.
To state the formulas conveniently, we present the formulas for the  function $\tilde{\psi}_\ell^{(m)}(\cdot):={\psi}_\ell^{(m)}(\cdot+\frac{j_m}{2})$.  Note that the Fourier transform of $\tilde{\psi}_\ell^{(m)}$  is
\begin{equation}\label{eq:de}
\widehat{\tilde{\psi}}_\ell^{(m)}(\omega)= i^\ell \sqrt{{m\choose \ell}}\frac{\cos^{m-\ell}(\frac{\omega}{4})\sin^{m+\ell}(\frac{\omega}{4})}{(\frac{\omega}{4})^m}.
\end{equation}
We note that the formulas presented in this subsection are used to calculate the function value and the derivative of $\tilde{\psi}_\ell^{(m)}$. When $m$ is even, ${\psi}_\ell^{(m)}\equiv \tilde{\psi}_\ell^{(m)}$. When $m$ is odd, one  can obtain those of the function ${\psi}_\ell^{(m)}$ by the half-translation of $\tilde{\psi}_\ell^{(m)}$.  Hence, the formulas given in this subsection also  work for ${\psi}_\ell^{(m)}$ with a proper shift.

Next, we present the recurrence relations of  framelets $\tilde{\psi}_\ell^{(m)}$.

\begin{thm}
Let $m\in \N$  be given and $1\leq \ell \leq m$ and let the  framelet $\tilde{\psi}_\ell^{(m)}$ derived from B-spline of order $m$
 be given via its  Fourier transform as (\ref{eq:de}).
 Then,   we have the following  recurrence  formula  between
  $\tilde{\psi}_{\ell}^{(m+1)}$ and  $\tilde{\psi}_{\ell}^{(m)}$:  for $1\leq \ell\leq m$:
\begin{eqnarray}\label{eq:re2}
 & &\tilde{\psi}_\ell^{(m+1)}(x)=\sqrt{\frac{m+1}{m+1-\ell}}\\
 & &\left( \frac{2x+m+1}{2m}\tilde{\psi}_{\ell}^{(m)}\left(x+\frac{1}{2}\right)+\break \frac{m+1-2x}{2m}\tilde{\psi}_{\ell}^{(m)}\left(x-\frac{1}{2}\right)
+\frac{\ell}{m}\tilde{\psi}_{\ell}^{(m)}(x)\right);\nonumber
\end{eqnarray}
 the recurrence formula   between  $\tilde{\psi}_{m+1}^{(m+1)}$ and  $\tilde{\psi}_{m}^{(m)}$ is:
\begin{equation}\label{eq:re1}
\tilde{\psi}_{m+1}^{(m+1)}(x)=\frac{2x+m+1}{2m}\tilde{\psi}_m^{(m)}\left(x+\frac{1}{2}\right)+
\frac{2x-m-1}{2m}\tilde{\psi}_m^{(m)}\left(x-\frac{1}{2}\right)
-\frac{2x}{m}\tilde{\psi}_m^{(m)}(x).
\end{equation}
\end{thm}
\begin{proof}
We firstly prove (\ref{eq:re2}) which is done in Fourier domain.
Note that
$$
\Dev{\omega}{}{\widehat{\tilde{\psi}}_{\ell}^{(m)}(\omega)}=-i\int_{-\infty}^\infty x\tilde{\psi}_\ell^{(m)}(x)e^{-i\omega x}dx,
$$
which implies that the Fourier transform of function $g_\ell(x):=x\tilde{\psi}_{\ell}^{(m)}(x)$ is
\begin{equation}\label{eq:fourf}
\hat{g}_\ell(\omega)= i^{\ell+1}\cdot 4^{m-1}\cdot \sqrt{m\choose \ell}  \cos(\frac{\omega}{4})^{m-\ell-1}\sin(\frac{\omega}{4})^{m+\ell-1} \frac{m\omega\cos(\frac{\omega}{2})-2m\sin(\frac{\omega}{2})+\ell\cdot\omega }{\omega^{m+1}}.
\end{equation}
Note that
$$
\frac{2x+m+1}{2m}\tilde{\psi}_\ell^{(m)}(x+\frac{1}{2})=\frac{1}{m}g_\ell(x+\frac{1}{2})+\frac{1}{2}\tilde{\psi}_\ell^{(m)}(x+\frac{1}{2})
$$
and
$$
\frac{m+1-2x}{2m}\tilde{\psi}_\ell^{(m)}(x-\frac{1}{2})=\frac{1}{2}\tilde{\psi}_\ell^{(m)}(x-\frac{1}{2})-\frac{1}{m}g_\ell(x-\frac{1}{2}).
$$
A simple manipulation shows that the Fourier transform of the right hand side of
(\ref{eq:re2}) becomes
\begin{eqnarray*}
& &\sqrt{\frac{m+1}{m+1-\ell}}(\frac{1}{m}\exp(\frac{i\omega}{2})\hat{g}_\ell(\omega)+\frac{1}{2}\cdot (\exp(\frac{i\omega}{2})+\exp(-\frac{i\omega}{2}))
\widehat{\tilde{\psi}}_\ell^{(m)}(\omega)\\
& &-\frac{1}{m} \exp(-\frac{i\omega}{2})\hat{g}_\ell(\omega)+\frac{\ell}{m}\widehat{\tilde{\psi}}_\ell^{(m)}(\omega))\\
&=&i^\ell \sqrt{{m+1\choose \ell}}\frac{\cos^{m+1-\ell}(\frac{\omega}{4})\sin^{m+1+\ell}(\frac{\omega}{4})}{(\frac{\omega}{4})^{m+1}}=\widehat{\tilde{\psi}}_\ell^{(m+1)}(\omega).
\end{eqnarray*}
This proves (\ref{eq:re2}). Similarly, the Fourier transform of the right side of (\ref{eq:re1}) is
\begin{eqnarray*}
& &\frac{1}{m}\left(\exp{(\frac{i\omega}{2})}(\hat{g}_m(\omega)+\frac{m}{2}\widehat{\tilde{\psi}}_m^{(m)}(\omega))+
\exp{(-\frac{i\omega}{2})}(\hat{g}_m(\omega)-\frac{m}{2}\widehat{\tilde{\psi}}_m^{(m)}(\omega))-2\hat{g}_m(\omega)
\right)\\
&=&i^{m+1}\frac{\sin^{2m+2}(\frac{\omega}{4})}{(\frac{\omega}{4})^{m+1}}=\widehat{\tilde{\psi}}_{m+1}^{(m+1)}(\omega),
\end{eqnarray*}
which proves (\ref{eq:re1}).
\end{proof}

 Furthermore,  combining  (\ref{eq:re2}) and (\ref{eq:re1}), we have a recurrence algorithm for efficiently computing $\tilde{\psi}_\ell^{(m)}, \ell=1,\ldots,m$. When $\ell <m$, we can use (\ref{eq:re2}) to compute  $\tilde{\psi}_\ell^{(m)}$ by $\tilde{\psi}_\ell^{(m-1)}$; we can use (\ref{eq:re1}) to compute $\tilde{\psi}_\ell^{(\ell)}$ by $\tilde{\psi}_{\ell-1}^{(\ell-1)}$. Hence, we finally can reduce the computation of $\tilde{\psi}_\ell^{(m)}$ to  that of  $\tilde{\psi}_1^{(1)}$.
Note that the function $\tilde{\psi}_1^{(1)}$ is Haar wavelet and
$$
\tilde{\psi}_{1}^{(1)}(x)=\begin{cases}
1, & \hbox{ if $x\in [-1/2,0)$ },\\
-1,  & \hbox{ if $x\in [0,1/2]$ },\\
0 ,& \hbox{ if $|x|>1/2$}.
\end{cases}
$$

We next show the method for computing $\tilde{\psi}_2^{(4)}$ by a table. In the following table, for the notation $\rightarrow$, we use
the formula (\ref{eq:re2}), while for the notation $ \searrow$, we use (\ref{eq:re1}):
\begin{center}
\begin{tabular}{p{10pt}cccccc}
$B_1$ &   & $B_2$ &  & $B_3$ & & $B_4$   \\
$\tilde{\psi}_1^{(1)}$&  & $\tilde{\psi}_1^{(2)}$& & $\tilde{\psi}_1^{(3)}$& & $\tilde{\psi}_1^{(4)}$ \\
                    & ${\searrow}$ & $\tilde{\psi}_2^{(2)}$ &$\rightarrow$ &$\tilde{\psi}_2^{(3)}$ & $\rightarrow$ &$\tilde{\psi}_2^{(4)}$ \\
            &  &              & &$\tilde{\psi}_3^{(3)}$ & &$\tilde{\psi}_3^{(4)}$ \\
            &  &              & &             & &$\tilde{\psi}_4^{(4)}$ \\
\end{tabular}.
\end{center}

Using the method, we compute $B_5$ and corresponding framelets (See Figure 1).
\begin{figure}[!ht]
\begin{center}
\epsfxsize=12cm\epsfbox{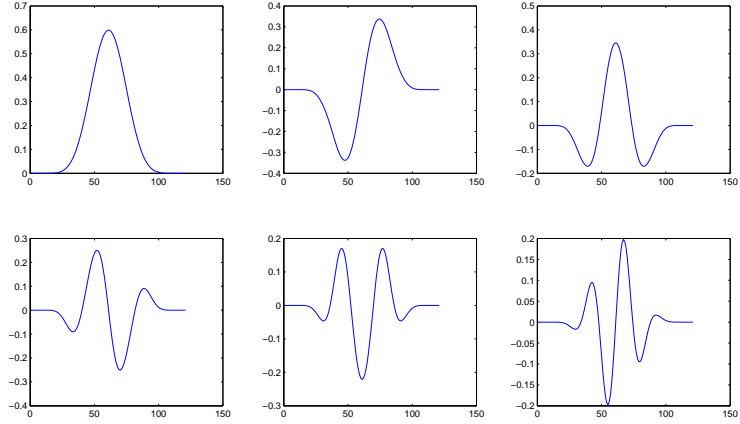}
\end{center}
\caption{ The $B_5$ and corresponding framelets. }
\end{figure}

Next, we give the recurrence formula for the derivatives  of $\tilde{\psi}_\ell^{(m)}$:
\begin{thm}
Let $m\in \N$ be given and $1\leq \ell \leq m$,  and let the framelet $\tilde{\psi}_{\ell}^{(m)}$  derived from B-spline of order $m$ be defined
  by its Fourier transform as (\ref{eq:de}).
When $1\leq \ell\leq m-1$,  we have
\begin{equation}\label{eq:de1}
\Dev{x}{}{\tilde{\psi}_\ell^{(m)}(x)}=\sqrt{\frac{m}{m-\ell}} \left(\tilde{\psi}^{(m-1)}_\ell\left(x+\frac{1}{2}\right)-\tilde{\psi}^{(m-1)}_\ell\left(x-\frac{1}{2}\right)\right).
\end{equation}
When $\ell=m$, we have
\begin{equation}\label{eq:de2}
\Dev{x}{}{\tilde{\psi}_m^{(m)}(x)}= \tilde{\psi}_{m-1}^{(m-1)}\left(x+\frac{1}{2}\right)+\tilde{\psi}_{m-1}^{(m-1)}\left(x-\frac{1}{2}\right)-2\tilde{\psi}_{m-1}^{(m-1)}(x).
\end{equation}
\end{thm}
\begin{proof}
We prove  (\ref{eq:de1}) here  while (\ref{eq:de2}) can be proven similarly.
A simple calculation shows that  the Fourier transform of the right side of (\ref{eq:de1}) is
\begin{eqnarray}\label{eq:proofde}
& &\sqrt{\frac{m}{m-\ell}}\cdot \sqrt{{m-1\choose \ell}}\cdot i^\ell \cdot \frac{\cos^{m-1-\ell}(\frac{\omega}{4})\sin^{m-1+\ell}(\frac{\omega}{4})}{(\frac{\omega}{4})^{m-1}}(e^{i\frac{\omega}{2}}-e^{-i\frac{\omega}{2}})\nonumber\\
& &= 4i^{\ell+1} \sqrt{m\choose \ell} \frac{\cos^{m-\ell}(\frac{\omega}{4}) \sin^{m+\ell}(\frac{\omega}{4})}{(\frac{\omega}{4})^{m-1}}.
\end{eqnarray}
Note that   the Fourier transform of $\Dev{x}{}{\tilde{\psi}_\ell^{(m)}(x)}$ is
\begin{equation}\label{eq:proofde1}
4\cdot i^{\ell+1}\cdot \sqrt{m\choose \ell} \frac{\cos^{m-\ell}(\frac{\omega}{4}) \sin^{m+\ell}(\frac{\omega}{4})}{(\frac{\omega}{4})^{m-1}}.
\end{equation}
Combining (\ref{eq:proofde}) and (\ref{eq:proofde1}), we conclude  (\ref{eq:de1}).
\end{proof}
\begin{remark}
Note that   $\tilde{\psi}_0^{(m)}=B_m$. If we  take  $\ell=0$ in \eqref{eq:re2},
the recurrence relation (\ref{eq:re2}) is reduced to  \eqref{eq:Bsplinere}, which is the recurrence formula for B-splines. Similarly,
if we take $\ell=0$ in \eqref{eq:de1}, then \eqref{eq:de1} is reduced to the derivative formula of B-splines  \eqref{eq:Bsplinedev}.
\end{remark}

\subsection{Representing  $\psi_\ell^{(m)}$ as the $\ell$th  derivative of a  univariate box spline}

We first recall the definition  of box splines.
The {\em  box spline} $B(\cdot |\Xi)$ associated with a matrix $\Xi\in \R^{s\times n}$
is the distribution given by the rule (see \cite{deboorbook})
\begin{equation}\label{Eq:boxspline}
\int_{\R^s}\!\!\! B(x |\Xi)\varphi (x)dx = \int _{[-\frac{1}{2},\frac{1}{2})^n}\!\! \varphi
(\Xi u)du,\,\, \text{ for all }\varphi \in {\mathscr D}(\R^s),
\end{equation}
where ${\mathscr D}(\R^s)$ is the test function space.
The box spline can be consider as a volume function of the section of unit cubes
(see \cite{deboorbook,xu1,xu2}). If we take $\Xi=(1,1,\ldots,1)\in \R^{1\times m}$,
 then the box spline $B(\cdot|\Xi)$ is reduced to a B-spline of order $m$.
  In the following theorem, we show that the B-spline framelet can be
   considered as the higher order derivative of a box spline up to a constant.
\begin{thm}\label{th:framederiv}
Let $m\in \N$ be given and $1\leq \ell \leq m$.  Suppose that the framelet $\psi_{\ell}^{(m)}$ is
 defined by its Fourier transform in (\ref{eq:pside}).
Set
$$
\Xi_{m,\ell}:=[\underbrace{1,\ldots,1}_{m-\ell},\underbrace{\frac{1}{2},\ldots,\frac{1}{2}}_{2\ell}].
$$
Then
\begin{equation}\label{eq:intoffram}
\psi_\ell^{(m)}(x)= {\sqrt{m\choose \ell}}\cdot \frac{1}{{4^\ell}} \cdot \Dev{x}{\ell}{B(x-\frac{j_m}{2}|\Xi_{m,\ell})},
\end{equation}
where $j_m$ is defined in \eqref{defjm}.
In particular, $\psi_m^{(m)}$ is the $m$-order derivative of $B_{2m}(2\cdot-j_m)/4^m$, where $B_{2m}$ is the B-spline of order $2m$.
\end{thm}
\begin{proof}  This again is proven in Fourier domain.
It follows from the definition of box splines  (\ref{Eq:boxspline}) that the Fourier transform  of the box spline ${B}(\cdot|\Xi_{m,\ell})$ is:
$$
\hat{B}(\omega|\Xi_{m,\ell})= \left(\frac{\sin {\frac{\omega}{2}}}{\frac{\omega}{2}}\right)^{m-\ell} \left(\frac{\sin {\frac{\omega}{4}}}{\frac{\omega}{4}}\right)^{2\ell}.
$$
Then the Fourier transform of
$$
{\sqrt{m\choose \ell}}\cdot \frac{1}{4^\ell}\cdot \Dev{x}{\ell}{B(x-\frac{j_m}{2}|\Xi_{m,\ell})}
$$
can be computed as:
\begin{eqnarray*}
& &{\sqrt{m\choose \ell}}\cdot \frac{1}{4^\ell}\cdot e^{-ij_m\omega/2}(i\omega)^\ell\hat{B}(\omega|\Xi_{m,\ell})\\
&=&i^\ell  \sqrt{m\choose \ell}e^{-ij_m\omega/2} \left(\frac{\sin \frac{\omega}{2}}{\frac{\omega}{2}}\right)^{m-\ell} \frac{\left(\sin\frac{\omega}{4}\right)^{2\ell}}{\left(\frac{\omega}{4}\right)^\ell}\\
&=&i^\ell  \sqrt{m\choose \ell}e^{-ij_m\omega/2} \left(\frac{\sin \frac{\omega}{4}\cos\frac{\omega}{4}}{\frac{\omega}{4}}\right)^{m-\ell} \frac{(\sin\frac{\omega}{4})^{2\ell}}{(\frac{\omega}{4})^\ell}
\\
&=&i^\ell\sqrt{m\choose \ell} e^{-ij_m\omega/2} \frac{\cos^{m-\ell}(\frac{\omega}{4})\sin^{m+\ell}(\frac{\omega}{4})}{(\frac{\omega}{4})^m}\\
&=&\hat{\psi}_\ell^{(m)}(\omega),
\end{eqnarray*}
which proves (\ref{eq:intoffram}). According to the definition of box splines, we have
$$
B_{2m}(2x-j_m)=B(x-\frac{j_m}{2}|\Xi_{m,m}).
$$
And hence, $\psi_m^{(m)}$ is the $m$-order derivative of $B_{2m}(2\cdot-j_m)/4^m$.
\end{proof}

Theorem \ref{th:framederiv} states  that each framelet of vanishing moment order of
$\ell$ is the $\ell$th derivative of a univariate box spline whose support is the same
as the framelet and whose Fourier transform dose not vanish at origin. This fact is
used in \cite{CDOS11} to discretize differential operators by using framelets.
 We illustrate here  for the case where $m$ is even how the framelet coefficients
  can be viewed as the samples of a difference of a given function which can
 be used to approximate the derivative of this function when it is smooth.
  Recall that,  if $f\in L_2(\R)$, then
$$
f=\sum_{\ell=1}^m\sum_{k,n\in \Z}\innerp{f,\psi_{\ell,n,k}^{(m)}}\psi_{\ell,n,k}^{(m)}.
$$
We next show that the coefficients $\innerp{f,\psi_{\ell,n,k}^{(m)}}$
is the $\ell$-order difference of a discretization of $f$ up to a constant.
We set
$$
T_nf(x):=2^{\frac{n}{2}}\biginnerp{f\left(\frac{x+\cdot}{2^n}\right),B(\cdot|\Xi'_{m,\ell})},
$$
where
$$
\Xi'_{m,\ell}:=[\underbrace{1,\ldots,1}_{m-\ell},\underbrace{\frac{1}{2},\ldots,\frac{1}{2}}_{\ell}].
$$
Then $T_nf$ can be considered as a discretization of $f$.
We define the difference operator by
$$
\Delta_{\frac{1}{4}}T_nf(x):=T_nf\left(x+\frac{1}{4}\right)-T_nf\left(x-\frac{1}{4}\right).
$$
Now we can give an explanation of the coefficient of $\innerp{f,\psi_{\ell,n,k}^{(m)}}$ by the difference operator $\Delta_{\frac{1}{4}}$.
In fact, we have
\begin{eqnarray*}
& &\innerp{f,\psi_{\ell,n,k}^{(m)}}\\
&=&(-1)^\ell {\sqrt{m\choose \ell}}\cdot \frac{1}{4^\ell} \cdot 2^{n/2}\cdot \frac{1}{2^{n\ell}}\cdot \biginnerp{{f}^{(\ell)}\left(\frac{\cdot+k}{2^n}\right),B(\cdot|\Xi_{m,\ell})}\\
&=&(-1)^\ell {\sqrt{m\choose \ell}}\cdot \frac{1}{4^\ell} \cdot 2^{n/2}\cdot \frac{1}{2^{n\ell}}\cdot \int_{[-1/2,1/2)^{m+\ell}} f^{(\ell)}\left(\frac{\Xi_{m,\ell}u+k}{2^n}\right)du\\
&=& (-1)^\ell {\sqrt{m\choose \ell}}\cdot \frac{1}{2^\ell} \cdot \Delta_{\frac{1}{4}}^\ell T_nf(k).
\end{eqnarray*}
Here,  the first identity follows from the result of Theorem \ref{th:framederiv} and the
 integration by parts. The second relation is obtained by the definition of box splines.
\begin{remark}
Theorem \ref{th:framederiv} shows that one can obtain the B-spline framelet
by calculating the derivative of box splines, which provides a new path to
 construct spline framelets. We hope to construct multivariate spline
  framelets by calculating the derivative of some relevant box splines in future work.
\end{remark}

\section{The asymptotic property of B-spline framelets}

\subsection{The asymptotic convergence of univariate box splines}
Up to the normalization,  the B-spline tends to
 Gaussian function pointwise and in all $L_p$ norms with $2\leq p <+\infty$
 as the order tends to infinity (see \cite{Unser92}).   Motivated by the results,
  in this subsection,   we investigate the  asymptotic convergence of
  univariate box splines, which is helpful to understand the
  convergence of  $\psi_\ell^{(m)}$, with  $\psi_\ell^{(m)}$
  being the $\ell$-order   derivative of a box spline up to a constant.

To state the results conveniently,  throughout the rest of this paper,
we shall use the notation $X\lesssim_{a,b,\ldots} Y$ to refer to the
 inequality $X\leq C\cdot Y$, where the constant $C$ may depend on
 $a,b,\ldots,$ but no other variable. In the next theorem,
 we show that the normalized box splines converge uniformly to a Gaussian function.

\begin{thm}\label{th:boxconv}  For each $k\in \N$,
Let
$$
\Xi_k:=[a_1^{(k)},\ldots,a_k^{(k)}]\in \R^{1\times k},
$$
be a given a set of points with  $a_j^{(k)}>0, ,j=1,\ldots,k$.
Let $B(\cdot|\Xi_k)$ be the box spline associate with  $\Xi_k$.
Assume that
\begin{equation}\label{eq:sum1}
\| \Xi_k\|^2_{2}=\sigma^2+\epsilon_k,
\end{equation}
with $\sigma\in \R$ is a fixed constant and $\lim_{k\rightarrow \infty}\epsilon_k=0$,
and  assume that
\begin{equation}\label{eq:boundc1c2}
c_1\leq \frac{\max_{1\leq j\leq k} a_j^{(k)}}{\min_{1\leq j\leq k} a_j^{(k)}}\leq c_2
\end{equation}
where $c_1$ and $c_2$ are  fixed positive constants independent of $k$. Then,
\begin{equation}\label{eq:convrate}
\max_{x}\bigabs{{\sqrt{\frac{6}{{\pi}\sigma^2}}\exp{\left(-\frac{6x^2}{\sigma^2}\right)}-B(x|\Xi_k)}}
\lesssim_{c_1,c_2} \frac{(\ln k)^3}{k}+\abs{\epsilon_k}\cdot \abs{\ln(\abs{\epsilon_k})}\cdot \ln (k).
\end{equation}
\end{thm}
In order to prove Theorem \ref{th:boxconv}, we need the following  technical lemma about the Fourier transform of the box spline
$B(\cdot|\Xi_k)$:
\begin{lem}\label{le:boxconv}
Under the conditions of Theorem \ref{th:boxconv},
\begin{equation}\label{eq:sincgauss}
\max_{\omega}\bigabs{{f_k(\omega)-\exp\left(-\frac{(\sigma\omega)^2}{24}\right)}} \lesssim_{c_1,c_2} \frac{(\ln k)^2}{k}+\abs{\epsilon_k}\cdot \abs{\ln\abs{\epsilon_k}},
\end{equation}
where
$$
f_k(\omega):= \hat{B}(\omega|\Xi_k)=\prod_{j=1}^k\sinc\left(\frac{a_j^{(k)}\omega}{2}\right).
$$
\end{lem}

\begin{proof}
Without loss of generality, we suppose that for each fixed $k$
$$
0<a_1^{(k)}\leq a_2^{(k)}\leq \cdots \leq  a_k^{(k)}.
$$
Then \eqref{eq:sum1} and (\ref{eq:boundc1c2}) imply that
$$
\frac{1}{\sqrt{k}}\lesssim_{c_1,c_2} a_1^{(k)}\leq a_k^{(k)} \lesssim_{c_1,c_2} \frac{1}{\sqrt{k}}.
$$

We firstly consider the case $\abs{\omega} \geq {\pi}/{a_k^{(k)}}$. Note that $\sinc(\cdot)$ is a monotone decreasing function in
$[0,\pi]$ and
$$
\abs{\sinc (\omega)}\leq \frac{1}{\pi},\qquad \mbox{for }\abs{\omega} \geq \pi.
$$
 Then, we have
\begin{eqnarray*}
& &\max_{\abs{\omega}\geq \pi/a_k^{(k)}} \abs{f_k(\omega)}=\max_{\abs{\omega}\geq \pi/a_k^{(k)}} \prod_{j=1}^k\,\,\bigabs{\sinc\left(\frac{a_j^{(k)}\omega}{2}\right)}\\
&\leq&\max\{\frac{1}{\pi^k}, \left(\sinc\left(\frac{\pi}{2}\frac{a_1^{(k)}}{a_k^{(k)}}\right)\right)^k\} \lesssim_{c_1,c_2} \beta^k,
\end{eqnarray*}
where $\beta<1$ is a positive constant.
And hence,  when $\abs{\omega}\geq {\pi}/{a_k^{(k)}}$,
\begin{equation}\label{eq:3int}
\bigabs{f_k(\omega)-\exp\left(-\frac{(\sigma\omega)^2}{24}\right)}\leq \abs{f_k(\omega)}+\exp\left(-\frac{(\sigma\omega)^2}{24}\right) \lesssim_{c_1,c_2} \frac{1}{k},
\end{equation}
which implies (\ref{eq:sincgauss}).

We next consider the case where $\abs{\omega}\leq \pi/{a_k^{(k)}}$.
 Taylor expansion shows that, when $\abs{\omega}\leq \pi/{a_k^{(k)}}$,
\begin{equation}\label{eq:taylor}
\ln f_k(\omega)=\sum_{j=1}^k\ln \left(\sinc\left(\frac{a_j^{(k)}\omega}{2}\right)\right)=-\left(\|\Xi_k\|_2^{2} \cdot \frac{\omega^2}{24}+S(\omega)\right),
\end{equation}
where
$$
S(\omega)=\frac{\|\Xi_k\|_4^4\cdot \omega^4}{2880}+\frac{\|\Xi_k\|_6^6\cdot \omega^6}{181440}+\cdots
$$
 is a uniformly convergent series on $\abs{\omega}\leq  \pi/a_k^{(k)}$.
By (\ref{eq:taylor}), we now obtain that, when $\abs{\omega}\leq  \pi/a_k^{(k)}$,
$$
f_k(\omega)=\prod_{j=1}^k\sinc\left(\frac{a_j^{(k)}\omega}{2}\right)=\exp\left(-\frac{(\sigma\omega)^2}{24}\right)\cdot
\exp\left(-\frac{\epsilon_k\omega^2}{24}\right)\cdot \exp(- S(\omega)).
$$
Hence,
\begin{eqnarray}
& &\bigabs{f_k(\omega)-\exp\left(-\frac{\sigma^2\omega^2}{24}\right)}\label{eq:pr11}\\
&\leq &\exp(-S(\omega))\exp\left(-\frac{\sigma^2\omega^2}{24}\right)\cdot\abs{(\exp\left(-\frac{\epsilon_k\omega^2}{24}\right)-1)}
\nonumber\\
& &+\exp\left(-\frac{\sigma^2\omega^2}{24}\right)\cdot \abs{\exp(-S(\omega))-1}.\nonumber
\end{eqnarray}
Once, we prove that
\begin{equation}\label{eq:pr12}
\exp\left(-\frac{\sigma^2\omega^2}{24}\right)\cdot \bigabs{\exp\left(-\frac{\epsilon_k\omega^2}{24}\right)-1}\lesssim \abs{\epsilon_k}\cdot \abs{\ln\abs{\epsilon_k}}
\end{equation}
and
\begin{equation}\label{eq:pr13}
\exp\left(-\frac{\sigma^2\omega^2}{24}\right)\cdot\abs{\exp(-S(\omega))-1} \lesssim \frac{(\ln k)^2}{k}.
\end{equation}
Then, combining (\ref{eq:pr11}), (\ref{eq:pr12}) and (\ref{eq:pr13}), we obtain  \eqref{eq:sincgauss}.

It remains to prove (\ref{eq:pr12}) and (\ref{eq:pr13}).
We first prove that
$$
\exp\left(-\frac{\sigma^2\omega^2}{24}\right)\cdot \abs{\exp\left(-\frac{\epsilon_k\omega^2}{24}\right)-1}\,\,\lesssim \,\, \abs{\epsilon_k}\cdot \abs{\ln\abs{\epsilon_k}}.
$$
By Taylor's expansion, when ${\omega}^2\leq 24\cdot \abs{\ln \abs{\epsilon_k}}/\sigma^2$,
$$
\exp\left(-\frac{\sigma^2\omega^2}{24}\right)\abs{\exp\left(-\frac{\epsilon_k\omega^2}{24}\right)-1} \leq \abs{(\exp\left(-\frac{\epsilon_k\omega^2}{24}\right)-1)} \lesssim \abs{\epsilon_k}\cdot \abs{\ln\abs{\epsilon_k}};
$$
when ${\omega}^2\geq 24\cdot \abs{\ln \abs{\epsilon_k}}/\sigma^2$,
$$
\exp\left(-\frac{\sigma^2\omega^2}{24}\right)\bigabs{\exp\left(-\frac{\epsilon_k\omega^2}{24}\right)-1} \leq  2\exp\left(-\frac{\sigma^2\omega^2}{24}\right)\lesssim \abs{\epsilon_k}.
$$
This gives  \eqref{eq:pr12}.

 We next prove that (\ref{eq:pr13}).
Note that, when $\abs{\omega}\leq \sqrt{24\ln k}/\sigma$,
\begin{equation}\label{eq:1int}
 \exp\left(-\frac{\sigma^2\omega^2}{24}\right)\bigabs{1-\exp(-S(\omega))}\lesssim \abs{1-\exp(-S(\omega))}\lesssim_{c_1,c_2} \frac{(\ln k)^2}{k}.
\end{equation}
When  $\sqrt{24 \ln k}/\sigma \leq \abs{\omega}\leq \pi /a_k^{(k)}$, we have
$$
\exp\left(-\frac{\sigma^2\omega^2}{24}\right) \leq \frac{1}{k},
$$
which implies that
\begin{equation}\label{eq:sincexpbound}
\exp\left(-\frac{\sigma^2\omega^2}{24}\right)\abs{\exp(-S(\omega))-1}\leq \frac{2}{k} \lesssim \frac{(\ln k)^2}{k}.
\end{equation}
Combining \eqref{eq:1int} and \eqref{eq:sincexpbound},  one derives (\ref{eq:pr13}).
\end{proof}

\begin{proof}[{\bf Proof of Theorem \ref{th:boxconv}}]
Note that
\begin{eqnarray*}
\frac{1}{2\pi}\int_{-\infty}^\infty \exp\left(-\frac{\sigma^2\omega^2}{24}\right)\exp(i\omega x)d\omega&=&\sqrt{\frac{6}{\pi\sigma^2}}\exp{\left(-\frac{6x^2}{\sigma^2}\right)},\\
\frac{1}{2\pi}\int_{-\infty}^\infty f_k(\omega)\exp(i\omega x)d\omega&=&B(x|\Xi_k).
\end{eqnarray*}
Then
\begin{eqnarray*}
& &\max_x\bigabs{\sqrt{\frac{6}{{\pi}\sigma^2}}\exp{\left(-\frac{6x^2}{\sigma^2}\right)}-B(x|\Xi_k)}\lesssim \int_{-\infty}^\infty \bigabs{\exp\left(-\frac{\sigma^2\omega^2}{24}\right)-f_k(\omega)}d\omega\\
&=& \int_{\abs{\omega}\leq \frac{\sqrt{24}\ln k}{\sigma}} \bigabs{\exp\left(-\frac{\sigma^2\omega^2}{24}\right)-f_k(\omega)}d\omega \\
&+&\int_{\frac{\sqrt{24}\ln k}{\sigma}\leq \abs{\omega}\leq \frac{\pi}{a_k^{(k)}}} \abs{\exp\left(-\frac{\sigma^2\omega^2}{24}\right)-f_k(\omega)}d\omega\\
&+&\int_{\frac{\pi}{a_k^{(k)}}\leq \abs{\omega}\leq \frac{\pi}{a_1^{(k)}}} \abs{\exp\left(-\frac{\sigma^2\omega^2}{24}\right)-f_k(\omega)}d\omega+\int_{\abs{\omega}\geq \frac{\pi}{a_1^{(k)}}} \abs{\exp\left(-\frac{\sigma^2\omega^2}{24}\right)-f_k(\omega)}d\omega\\
&\lesssim & \frac{(\ln k)^3}{k}+\abs{\epsilon_k }\cdot \abs{\ln \abs{\epsilon_k}}\cdot \ln k+\frac{\sqrt{k}}{k^{\ln k}}+\beta^k\lesssim \frac{(\ln k)^3}{k}+\abs{\epsilon_k }\cdot \abs{\ln \abs{\epsilon_k}}\cdot \ln k,
\end{eqnarray*}
where $\beta =\max\{\frac{1}{\pi}, \sinc\left(\frac{\pi}{2}\frac{a_1^{(k)}}{a_k^{(k)}}\right)\}<1$.
Here, we use (\ref{eq:sincgauss}) to obtain that
$$
\int_{\abs{\omega}\leq \frac{\sqrt{24}\ln k}{\sigma}} \bigabs{f_k(\omega)-\exp\left(-\frac{\sigma^2\omega^2}{24}\right)}d\omega\,\lesssim \, \frac{(\ln k)^3}{k}+\abs{\epsilon_k }\cdot \abs{\ln \abs{\epsilon_k}}\cdot \ln k.
$$
Note that $a_k^{(k)}/\sqrt{k},k=1,2,\ldots,$  is a bounded sequence and
$$
\exp(-\frac{\sigma^2\omega^2}{24})\lesssim \frac{1}{k^{\ln k}}, \qquad \mbox{ for } \omega \geq \frac{\sqrt{24}\ln k}{\sigma}.
$$
Using a similar method as the proof of (\ref{eq:sincexpbound}), we have that
$$
\int_{\frac{\sqrt{24}\ln k}{\sigma}\leq \abs{\omega}\leq \frac{\pi}{a_k^{(k)}}} \abs{f_k(\omega)-\exp\left(-\frac{\sigma^2\omega^2}{24}\right)}d\omega\,\, \lesssim\,\, \frac{\sqrt{k}}{k^{\ln k}}.
$$
To estimate
$$
\int_{\abs{\omega}\geq \frac{\pi}{a_1^{(k)}}} \abs{f_k(\omega)-\exp(-\omega^2/24)}d\omega \leq \int_{\abs{\omega}\geq \frac{\pi}{a_1^{(k)}}} \abs{f_k(\omega)}d\omega +\int_{\abs{\omega}\geq \frac{\pi}{a_1^{(k)}}} \exp(-\omega^2/24)d\omega,
$$
 we use the facts of
$$
\int_{\abs{\omega}\geq \frac{\pi}{a_1^{(k)}}} \abs{f_k(\omega)}d\omega\leq \int_{\abs{\omega}\geq \frac{\pi}{a_1^{(k)}}} \left(\frac{2}{a_1^{(k)}\omega}\right)^kd\omega\lesssim \frac{1}{k}
$$
and
$$
\int_{\abs{\omega}\geq \frac{\pi}{a_1^{(k)}}} {\exp\left(-\frac{\omega^2}{24}\right)}d\omega\,\, \lesssim\,\, \frac{1}{k}.
$$
\end{proof}

Theorem \ref{th:boxconv} implies that the normalized box spline $B(\cdot|\Xi_{m,\ell})$ converges uniformly to a Gaussian function:
\begin{cor}\label{co:conv}
Suppose that
$$
\Xi_{m,\ell}=[\underbrace{1,\ldots,1}_{m-\ell},\underbrace{1/2,\ldots,1/2}_{2\ell}].
$$
Then, for each fixed $\ell$, $\sqrt{{{m-\frac{\ell}{2}}}}\cdot B( \sqrt{m-\frac{\ell}{2}}x |\Xi_{m,\ell})$  converges uniformly to $\sqrt{\frac{6}{\pi}} \exp(-6x^2)$, as $m\to\infty$.
\end{cor}
\begin{proof}
By the definition of box splines, we have
$$
\sqrt{m-\frac{\ell}{2}}\ B(\sqrt{m-\frac{\ell}{2}}x|\Xi_{m,\ell})=B(x| \frac{\Xi_{m,\ell}}{\sqrt{m-\ell/2}}).
$$
Note that, for each fixed $\ell$, $\|\frac{\Xi_{m,\ell}}{\sqrt{m-\ell/2}}\|_2^2=1$.
Then, Theorem \ref{th:boxconv} shows that the box spline
$B(x| \frac{\Xi_{m,\ell}}{\sqrt{m-\ell/2}})$, and hence $\sqrt{m-\frac{\ell}{2}}\ B(\sqrt{m-\frac{\ell}{2}}x |\Xi_{m,\ell})$,
converges uniformly to the Gaussian function $\sqrt{\frac{6}{\pi}}\exp(-6x^2)$.
\end{proof}

\begin{remark}

A well-known result is that $\sqrt{m}B_m(\sqrt{m}x)$
converges uniformly to $\sqrt{\frac{6}{\pi}} \exp(-6x^2)$ with  $m\to\infty$  (see \cite{Unser92, Bsplineuniform}). In fact,  the result can be considered as a particular case of Corollary \ref{co:conv}.
Note that
$$
\sqrt{{{m}}}\cdot B(\sqrt{m} x |\Xi_{m,0})=\sqrt{m}B_m(\sqrt{m}x).
$$
If we take $\ell=0$ in Corollary \ref{co:conv}, then we have that $\sqrt{{{m}}}\cdot B(\sqrt{m} x |\Xi_{m,0})$, and hence $\sqrt{m}B_m(\sqrt{m}x)$,
converges uniformly to $\sqrt{\frac{6}{\pi}} \exp(-6x^2)$ with  $m\to\infty$.
\end{remark}

\subsection{The asymptotic property  of B-spline framelets}
By changing variables, we can observe from
 Corollary \ref{co:conv}  that $B(x|\Xi_{m,\ell})$ is close to
$$
\sqrt{\frac{6}{\pi}}\frac{1}{\sqrt{m-{\ell}/{2}}} \exp\left(-\frac{12\cdot x^2}{2m-\ell}\right).
$$
Recall that  Theorem \ref{th:framederiv} says that
$$
\psi_\ell^{(m)}(x)= {\sqrt{m\choose \ell}}\cdot \frac{1}{{4^\ell}} \cdot \Dev{x}{\ell}{B(x-\frac{j_m}{2}|\Xi_{m,\ell})}.
$$
Motivated by the two observations, we consider the relation between $\psi_\ell^{(m)}(x)$ and the $\ell$th derivative of a Gaussian function $G(x)$, which is defined as
 $$
 G_{m,\ell}(x):=C_{m,\ell}\cdot \exp\left(- \frac{12 \cdot x^2}{2m-\ell} \right),
 $$
 where
 $$
C_{m,\ell}=\sqrt{\frac{6}{\pi}} \frac{\sqrt{m\choose \ell}}{\sqrt{m-\ell/2}\cdot 4^{\ell}}.
$$
Let
\begin{equation}\label{eq:Gml}
G_\ell^{(m)}(x):=\Dev{x}{\ell}{G_{m,\ell}}(x-\frac{j_m}{2}), \qquad \ell=1,\ldots,m,
\end{equation}
where $j_m$ is given in \eqref{defjm},
and
$$
G^{(m)}:=\{G_1^{(m)},\ldots,G_m^{(m)}\}.
$$
Then we have
\begin{thm}\label{th:strongconv}Let $m\in \N$ be given and $1\leq \ell \leq m$,  and  the framelet $\psi_{\ell}^{(m)}$ be
 defined by its Fourier transform in (\ref{eq:pside}) derived from B-spline of order $m$. Then,
$$
\max_{1\leq\ell \leq m}\max_{x\in \R}\abs{\psi_{\ell}^{(m)}(x)-G_\ell^{(m)}(x)}\lesssim \frac{(\ln m)^{5/2}}{m^{3/2}}.
$$
\end{thm}
In order to prove Theorem \ref{th:strongconv}, we need the following  two lemmas:

\begin{lem}\label{le:expless} The following inequality holds for
 every $\omega$, such that $\abs{\omega} \geq 20 \sqrt{\frac{\ln m}{m}}$:
$$
\max_{1\leq \ell\leq m} \sqrt{m\choose \ell}\cdot \frac{1}{4^\ell}\cdot \bigabs{\omega^\ell\cdot \exp\left(-(m-\frac{\ell}{2}) \frac{\omega^2}{24}\right)}\,\, \lesssim\,\, \frac{1}{m^4}.
$$
\end{lem}
\begin{proof}
For the convenient,  denote
$$
F_\ell(\omega):=\sqrt{m\choose \ell}\cdot \frac{1}{4^\ell}\cdot {\omega^\ell\cdot \exp\left(-(m-\frac{\ell}{2}) \frac{\omega^2}{24}\right)}
$$
and
$$
\omega_\ell:=\sqrt{\frac{24\cdot \ell}{2m-\ell}}.
$$
For each fixed $\ell\in [1,m]\cap \Z$, the function $F_\ell$ is increasing on the interval  $[0, \omega_\ell]$, while
 $F_\ell$ is decreasing on $[\omega_\ell, \infty)$. And hence $F_\ell$ arrives at the maximum value
 at $\omega_\ell$.
According to the inequality
 $$
 {m\choose \ell} \leq \left(\frac{m \cdot e}{\ell}  \right)^\ell,
 $$
we have
$$
\ln F_\ell(\omega_\ell)\leq -\frac{\ell}{2} \cdot \ln \frac{2(2m-\ell)}{3m}.
$$
Then, when
$$
\frac{2\ln m}{\ln 16-\ln 15}\leq \ell \leq \frac{2}{5} m,
$$
we have that
$$
\ln F_\ell(\omega_\ell)\leq -\frac{\ell}{2} \cdot \ln \frac{2(2m-\ell)}{3m}\leq -\frac{\ln\frac{4}{3}}{\ln \frac{16}{15}}\ln m \leq -4 \ln m.
$$
This implies that, whenever
$$
\frac{2\ln m}{\ln 16-\ln 15}\leq \ell \leq \frac{2}{5} m
$$
holds,  one has
\begin{equation}\label{eq:le61}
\max_{\omega}F_\ell(\omega)\,\, \leq\,\, F_\ell(\omega_\ell)\,\, \leq\,\, \frac{1}{m^4}.
\end{equation}
We next consider the  case where  $\frac{2m}{5}\leq \ell \leq \frac{4m}{5}$.
Using the inequality
$$
{m\choose \ell} \leq 2^m,
$$
one gets that
$$
\ln F_\ell(\omega_\ell)\leq \frac{\ln 2}{2} m-\frac{\ell}{2}\ln \frac{2 e (2m-\ell)}{3 \ell}.
$$
Therefore, when $\frac{2m}{5}\leq \ell \leq \frac{4m}{5}$, one has that
$$
\ln F_\ell(\omega_\ell)\leq \frac{\ln 2}{2} m-\frac{\ell}{2}\ln \frac{2 e (2m-\ell)}{3 \ell}\lesssim -m \lesssim -4\ln m,
$$
which implies that, whenever  $\frac{2m}{5}\leq \ell \leq \frac{4m}{5}$, the following holds:
\begin{equation}\label{eq:le62}
\max_\omega F_\ell(\omega)\leq F_\ell(\omega_\ell) \lesssim \frac{1}{m^4}.
\end{equation}
We now turn to the case $\frac{4}{5}m\leq \ell \leq m$.
For this case, we  apply the inequality
$$
{m\choose \ell} \leq \left(\frac{m\cdot e}{m-\ell}\right)^{m-\ell}
$$
to obtain that
$$
\ln F_\ell(\omega_\ell) \leq \frac{m-\ell}{2} \ln \frac{m\cdot e}{m-\ell}-\frac{\ell}{2}\ln \frac{(4m-2\ell)\cdot e}{3\ell}\lesssim -m \lesssim -4\ln m.
$$
Hence, when $\frac{4}{5}m \leq \ell \leq m$, we have that
\begin{equation}\label{eq:le63}
\max_\omega F_\ell(\omega)\leq F_\ell(\omega_\ell) \lesssim \frac{1}{m^4}.
\end{equation}
We finally consider the case where $1\leq \ell\leq \frac{2\ln m}{\ln 16-\ln 15}$. Note that, when $m$ is large enough,
we have $\omega_\ell \leq 20 \sqrt{\frac{\ln m}{m}}$. When $\abs{\omega}\geq 20\sqrt{\frac{\ln m}{m}}$, $F_\ell(\omega)$ reach the maximum value
at $\bar{\omega}_0 := 20\sqrt{\frac{\ln m}{m}}$. Then a simple calculation shows that, when $1\leq \ell\leq \frac{2\ln m}{\ln 16-\ln 15}$,
\begin{equation}\label{eq:le64}
\max_{\abs{\omega}\geq 20\sqrt{\frac{\ln m}{m}}} F_\ell(\omega)\leq  F_\ell(\bar{\omega}_0)\,\,\lesssim\,\, \frac{1}{m^4}.
\end{equation}
Combing \eqref{eq:le61}, \eqref{eq:le62}, \eqref{eq:le63} and \eqref{eq:le64},  we conclude the proof.
\end{proof}

\begin{lem}\label{le:minusless1m}  The following inequality holds for
 every $\omega \in \R$,
$$
\max_{1\leq \ell\leq m} \frac{\sqrt{m\choose \ell}}{4^\ell} \cdot \abs{\omega}^\ell\cdot \bigabs{\sinc\left(\frac{\omega}{2}\right)^{m-\ell}\sinc\left(\frac{\omega}{4}\right)^{2\ell}-\exp\left(-\frac{(m-\ell/2)\omega^2}{24}\right)} \lesssim \frac{\ln^2m}{m}.
$$
\end{lem}
\begin{proof}
For the  convenience,  we only provide the proof for the case where $\omega \geq 0$.  The proof of the other case is similar.
By Taylor expansion, when $0\leq \omega\leq \frac{3\pi}{2}$, we have
\begin{eqnarray*}
& &\sinc\left(\frac{\omega}{2}\right)^{m-\ell}\sinc\left(\frac{\omega}{4}\right)^{2\ell}\\
&=&\exp\left(-\frac{(m-\ell/2)\omega^2}{24}\right)\exp\left(-\frac{(m-7\ell/8)\omega^4}{2880}-O(\omega^6)\right).
\end{eqnarray*}
Then, for $ 20 \sqrt{\frac{\ln m}{m}} \leq \omega \leq \frac{3\pi}{2}$, we have
\begin{eqnarray*}
& &\frac{\sqrt{m\choose \ell}}{4^\ell}\cdot {\omega}^\ell\cdot \bigabs{\sinc\left(\frac{\omega}{2}\right)^{m-\ell}\sinc\left(\frac{\omega}{4}\right)^{2\ell}-\exp\left(-\frac{(m-\ell/2)\omega^2}{24}\right)}\\
&\lesssim&\frac{\sqrt{m\choose \ell}}{4^\ell}\cdot {\omega}^\ell\cdot  \exp\left(-\frac{(m-\ell/2)\omega^2}{24}\right)\left(1-\exp\left(-\frac{(m-7\ell/8)\omega^4}{2880}\right)\right)\\
&\leq& 2 \frac{\sqrt{m\choose \ell}}{4^\ell}\cdot {\omega}^\ell\cdot  \exp\left(-\frac{(m-\ell/2)\omega^2}{24}\right) \lesssim \frac{1}{m^4},
\end{eqnarray*}
where the last inequality is obtained by Lemma \ref{le:expless}.
Next,  when $0\leq\omega \leq 20 \sqrt{\frac{\ln m}{m}}$, note that
$$
1-\exp\left(-\frac{(m-7\ell/8)\omega^4}{2880}\right) \lesssim \frac{\ln^2 m}{m}
$$
and
$$
F_\ell(\omega)=\frac{\sqrt{m\choose \ell}}{4^\ell}\cdot {\omega}^\ell\cdot  \exp\left(-\frac{(m-\ell/2)\omega^2}{24}\right)
$$
is a bounded function.
Hence, for $0\leq\omega \leq 20 \sqrt{\frac{\ln m}{m}}$, we have
$$
\frac{\sqrt{m\choose \ell}}{4^\ell}\cdot {\omega}^\ell\cdot \bigabs{\sinc\left(\frac{\omega}{2}\right)^{m-\ell}\sinc\left(\frac{\omega}{4}\right)^{2\ell}-\exp\left(-\frac{(m-\ell/2)\omega^2}{24}\right)}
\lesssim \frac{\ln^2m}{m}.
$$
Finally, We consider the case  when $\omega \geq \frac{3\pi }{2}$.
We assert that, when $\omega \geq  \frac{3\pi}{2}$,  the following inequality holds:
\begin{equation}\label{eq:sincconv}
\max_{1\leq \ell\leq m}\sqrt{{m\choose \ell}}\cdot \abs{\sinc{\left(\frac{\omega}{2}\right)}}^{m-\ell}\cdot \abs{\sinc\left(\frac{\omega}{4}\right)}^\ell\,\,\leq\,\, \left(\frac{8\cdot e^{1/8}}{3\pi}\right)^m.
\end{equation}
With this assertion, we have
\begin{eqnarray*}
& &\frac{\sqrt{m\choose \ell}}{4^\ell}\cdot {\omega}^\ell\cdot \bigabs{\sinc\left(\frac{\omega}{2}\right)^{m-\ell}\sinc\left(\frac{\omega}{4}\right)^{2\ell}-\exp\left(-\frac{(m-\ell/2)\omega^2}{24}\right)}\\
&\leq & \frac{\sqrt{m\choose \ell}}{4^\ell}\cdot {\omega}^\ell\cdot
\left(\bigabs{\sinc\left(\frac{\omega}{2}\right)^{m-\ell}\sinc\left(\frac{\omega}{4}\right)^{2\ell}}+\bigabs{\exp\left(-\frac{(m-\ell/2)\omega^2}{24}\right)}\right)\\
&\leq & \sqrt{m\choose \ell} \left(\bigabs{\sinc\left(\frac{\omega}{2}\right)^{m-\ell}\sinc\left(\frac{\omega}{4}\right)^{\ell}}+\frac{1}{4^\ell}\cdot \bigabs{\omega^\ell \cdot\exp\left(-\frac{(m-\ell/2)\omega^2}{24}\right)}\right)\\
&\lesssim & \frac{1}{m^4}.
\end{eqnarray*}
Here, the last inequality is followed by \eqref{eq:sincconv} and Lemma \ref{le:expless}.
It remains to prove \eqref{eq:sincconv}.
Note that, when $\omega \geq \frac{3\pi}{2}$,
$$
\abs{\sinc{\left(\frac{\omega}{2}\right)}}^{m-\ell}\cdot \abs{\sinc\left(\frac{\omega}{4}\right)}^\ell \leq \frac{1}{(\omega/2)^{m-\ell}}\cdot \frac{1}{(\omega/4)^\ell} \leq
\frac{1}{(3\pi/4)^{m-\ell}}\cdot \frac{1}{(3\pi/8)^\ell}.
$$
Then we only need prove
$$
\max_{1\leq \ell\leq m}\sqrt{m\choose \ell}\cdot \frac{1}{(3\pi/4)^{m-\ell}}\cdot \frac{1}{(3\pi/8)^\ell}\lesssim  \left(\frac{8\cdot e^{1/8}}{3\pi}\right)^m.
$$
Applying  the inequality
${m\choose \ell} \leq \left(\frac{m\cdot e}{m-\ell} \right)^{m-\ell}$, we have that,
$$
\max_{1\leq \ell\leq m}\sqrt{m\choose \ell}\cdot \frac{1}{(3\pi/4)^{m-\ell}}\cdot \frac{1}{(3\pi/8)^\ell}\,\, \leq\,\, \left(\frac{8\cdot e^{1/8}}{3\pi}\right)^m.
$$
This proves that \eqref{eq:sincconv}.
\end{proof}

\begin{proof}[{\bf Proof of Theorem \ref{th:strongconv}}]
Let
$$
M(\omega):=\max_{1\leq \ell\leq m} \frac{\sqrt{m\choose \ell}}{4^\ell}\cdot \abs{\omega}^\ell\cdot \bigabs{\sinc\left(\frac{\omega}{2}\right)^{m-\ell}\sinc\left(\frac{\omega}{4}\right)^{2\ell}-\exp\left(-(m-\frac{\ell}{2})\frac{\omega^2}{24}\right)},
$$
and
\begin{eqnarray*}
& &I_1:=\{\omega\in \R: \abs{\omega} \leq 20\sqrt{\frac{\ln m}{m}}\},\, I_2:=\{\omega\in \R: 20\sqrt{\frac{\ln m}{m}} \leq \abs{\omega} \leq \frac{3\pi}{2} \}, \\
& &I_3:=\{\omega\in \R : \abs{\omega} \geq \frac{3\pi}{2}\}.
\end{eqnarray*}
Applying  Lemma \ref{le:minusless1m} and  Lemma \ref{le:expless}, we conclude  that
\begin{eqnarray*}
\int_{I_1} M(\omega)d\omega \lesssim  \frac{(\ln m)^{5/2}}{m^{3/2}},\qquad \int_{I_2} M(\omega)d\omega \lesssim  \frac{1}{m^4}
\end{eqnarray*}
respectively.
By an argument similar to that leading to \eqref{eq:sincconv},  we can obtain that there exists $0<\gamma <1$ such that
$$
 \int_{I_3} M(\omega)d\omega \lesssim  \gamma^m \lesssim  \frac{(\ln m)^{5/2}}{m^{3/2}}.
$$
This leads to that
\begin{eqnarray*}
& &\max_{1\leq\ell \leq m}\max_{x\in \R}\abs{\psi_{\ell}^{(m)}(x)-G_\ell^{(m)}(x)}\leq \int_{-\infty}^\infty M(\omega )d\omega \\
&=&\int_{I_1} M(\omega)d\omega +\int_{I_2}M(\omega)d\omega+\int_{I_3} M(\omega)d\omega \lesssim \frac{(\ln m)^{5/2}}{m^{3/2}}.
\end{eqnarray*}

\end{proof}

\begin{remark}
 It was proved in \cite{GGL} that,  for each fixed $\ell$, up to
  a normalization, a proper scaled
 $\psi_\ell^{(m)}$ uniformly converges to the $\ell$-order derivative
  of a  scaled Gaussian function with $m$ tending to infinity. Our result is in a different direction. In fact, we show that for sufficiently large  $m$ framelets $\psi_1^{(m)},\ldots,\psi^{(m)}_\ell,\ldots,\psi_m^{(m)}$ uniformly in $x$ and $\ell$  close to
 derivatives   of consecutive orders $1,\ldots,m$ of a scaled  Gaussian function whose scale depends on $m$.
\end{remark}

\section{Gaussian frame}

 Theorem \ref{th:strongconv} leads  us to consider  whether a wavelet system generated by a finite number of consecutive derivatives of a properly
 scaled Gaussian function forms a frame of $L_2(\R)$.
In this section, we show that the  frame property of $X(\psi_1^{(m)},\ldots,\psi_m^{(m)})$  can be transferred  to
that of $X(G_1^{(m)},\ldots,G_m^{(m)})$ where $G_\ell^{(m)}$ is defined as follows:  For each  fixed $m\in \N$,
 we consider the following rescaled Gaussian function
 $$
 G_{m,\ell}(x)=C_{m,\ell}\cdot \exp\left(- \frac{12 \cdot x^2}{2m-\ell} \right),
 $$
 where
 $$
C_{m,\ell}=\sqrt{\frac{6}{\pi}} \frac{\sqrt{m\choose \ell}}{\sqrt{m-\ell/2}\cdot 4^{\ell}};
$$
and
$$
G_\ell^{(m)}(x)=\Dev{x}{\ell}{G_{m,\ell}}(x-\frac{j_m}{2}), \qquad \ell=1,\ldots,m,
$$
where $j_m$ is given in \eqref{defjm}, and
$$
G^{(m)}=\{G_1^{(m)},\ldots,G_m^{(m)}\}.
$$

 Before stating the following main theorem of this section, we recall the definitions of the frame and { Bessel sequence}. A family $\{f_j\}_{j\in J}\subset L_2(\R)$ is called a {\em frame } with  bounds $A$ and $B$ if
$$
A\|f\|^2 \leq \sum_{j\in J}\abs{\innerp{f,f_j}}^2 \leq B\|f\|^2
$$
holds for all $f\in L_2(\R)$. If $A=B$, then $\{f_j\}_{j\in J}$ is called a {\em $A$-tight frame}. Moreover,
a family $\{f_j\}_{j\in J}\subset L_2(\R)$ is called a {\em Bessel sequence } with  a bound $R$ if
$$
 \sum_{j\in J}\abs{\innerp{f,f_j}}^2 \leq
 R\|f\|^2
$$
holds for all $f\in L_2(\R)$.

\begin{thm}\label{th:gaussframe} Let $X(G^{(m)})$ be the wavelet system generated by functions $G^{(m)}$. Then
$X(G^{(m)})$ is a frame system with frame bounds $A_m$ and $B_m$ for sufficiently large $m$. Furthermore, the frame is close to be tight as $m$ is sufficiently large. In fact, asymptotically, we have
$$
\lim_{m\rightarrow \infty}A_m=\lim_{m\rightarrow \infty}B_m=1.
$$
\end{thm}

To prove Theorem \ref{th:gaussframe}, we need the following theorem, which is proven in \cite{FZ}, together with  several lemmas.

\begin{thm}[\cite{FZ}]\label{th:FZ}
Let $\{f_j\}_{j\in J}$ be a frame of $L_2(\R)$ with bounds $A$ and $B$. Assume that
$\{g_j\}_{j\in J}\subset L_2(\R)$ is such that $\{f_j-g_j\}_{j\in J}$ is a Bessel sequence with a bound $R<A$.
Then $\{g_j\}_{j\in J}$ is a frame with bounds $A\left(1-\sqrt{\frac{R}{A}}\right)^2$ and $B\left(1+\sqrt{\frac{R}{B}}\right)^2$.
\end{thm}
Let
\begin{equation}\label{defphiml}
\phi_\ell^{(m)}:=\psi_\ell^{(m)}-{G}_\ell^{(m)},\, \ell=1,\ldots, m, \quad {\rm and }\quad \Phi^{(m)}:=\{\phi^{(m)}_1,\ldots,\phi^{(m)}_m\}.
\end{equation}
Since $X(\Psi^{(m)})$ is a tight frame with frame bound $1$,  to prove that $X(G^{(m)})$ is a  frame, according
to Theorem \ref{th:FZ}, we only need to show that $X(\Phi^{(m)})$ is
a Bessel sequence with a bound $R_m\rightarrow 0$.
 An estimate of the Bessel bound of a given a sequence  is provided in \cite{Ronshengramian}  that enables us to estimate the Bessel  bound of $X(\Phi^{(m)})$ (see also \cite{Daubechies}). Let
 \begin{equation}\label{eq:rm0}
R_m:=\sup_{1\leq \abs{\omega}\leq 2}\sum_{k\in \Z}\sum_{n\in \Z}\sum_{\ell=1}^m\, \abs{\hat{\phi}_\ell^{(m)}(2^n\omega)} \cdot \abs{\hat{\phi}_\ell^{(m)}(2^n\omega+2k\pi)}.
\end{equation}
 Then,  for arbitrary $f\in L^2(\R)$, the following inequality holds
$$
\sum_{\phi\in X(\Phi^{(m)})}\abs{\left<f,\phi\right>}^2\leq  R_m \|f\|_2^2,
$$
i.e. $R_m$ is the Bessel  up bound of the system $X(\Phi^{(m)})$.
Next, we estimate $R_m$. For this, we need the following lemmas.
\begin{lem}\label{le:upbound} Let $\hat{\phi}_\ell^{(m)}$ be the Fourier transform of $\phi_\ell^{(m)}$ defined in \eqref{defphiml}.
Then the following three estimates for $\abs{\hat{\phi}_\ell^{(m)}}$ holds:
\begin{enumerate}
\item[(i)]
$$
\abs{\hat{\phi}_\ell^{(m)}(\omega)}\,\, \leq\,\, \sqrt{m\choose \ell}\cdot \frac{2^{m+\ell+1}}{\abs{\omega}^m},\qquad \abs{\omega} \geq 20.
$$
\item[(ii)]
$$
\abs{\hat{\phi}_\ell^{(m)}(\omega)}\,\, \lesssim\,\,  \sqrt{m\choose \ell} \left(\frac{\omega}{4}\right)^\ell\cdot m\cdot \omega^4,\qquad \abs{\omega}\leq \sqrt{\frac{1}{m}}.
$$
\item[(iii)]
$$
\max_{1\leq \ell\leq m} \max_{\omega\in \R}\,\,\abs{\hat{\phi}_\ell^{(m)}(\omega)}\,\, \lesssim \,\, \frac{\ln^2m}{m}.
$$
\end{enumerate}
\end{lem}
\begin{proof} First a simple calculation leads to
\begin{eqnarray*}
\abs{\hat{\phi}_\ell^{(m)}(\omega)}&=&\abs{\hat{\psi}_\ell^{(m)}(\omega)-\hat{G}_\ell^{(m)}(\omega)}\\
&=&\Bigabs{\sqrt{m\choose \ell} \frac{\omega^\ell}{4^\ell} \left(\sinc\left(\frac{\omega}{2}\right)^{m-\ell}\sinc\left(\frac{\omega}{4}\right)^{2\ell}-\exp\left(-(m-\frac{\ell}{2}) \frac{\omega^2}{24}\right)\right)}.
\end{eqnarray*}
 For (i),  when $\abs{\omega}\geq 20$, a simple argument shows that
\begin{eqnarray*}
\abs{\frac{\omega^\ell}{4^\ell}}\exp\left(-(m-\frac{\ell}{2})\frac{\omega^2}{24}\right) \leq  \frac{2^{m+\ell}}{\abs{\omega}^m},\quad
  \bigabs{ (\frac{\omega}{4})^\ell\cdot \sinc(\frac{\omega}{2})^{m-\ell}\cdot \sinc(\frac{\omega}{4})^{2\ell}} \leq \frac{2^{m+\ell}}{\abs{\omega}^m},
\end{eqnarray*}
which implies that
$$
\abs{\hat{\phi}_\ell^{(m)}(\omega)}\,\, \leq\,\, \sqrt{m\choose \ell}\cdot \frac{2^{m+\ell+1}}{\abs{\omega}^m}.
$$
 For (ii), the Taylor expansion shows that, when $\abs{\omega}\leq \pi$,
$$
\ln \left(\sinc(\frac{\omega}{2})^{m-\ell}\sinc(\frac{\omega}{4})^{2\ell}\right)=-\left((m-\frac{\ell}{2})\frac{\omega^2}{24}+(m-\frac{7}{8}\ell)\frac{\omega^4}{2880}
+O(\omega^6)\right)
$$
Then, when $\abs{\omega}\leq \sqrt{\frac{1}{m}}$,
\begin{eqnarray*}
& &\abs{\sinc(\frac{\omega}{2})^{m-\ell}\sinc(\frac{\omega}{4})^{2\ell}-\exp\left(-(m-\frac{\ell}{2})\frac{\omega^2}{24}\right)}\\
&=&\exp\left(-(m-\frac{\ell}{2})\frac{\omega^2}{24}\right)\Bigabs{\exp\left(-(m-\frac{7}{8}\ell)\frac{\omega^4}{2880}-O(\omega^6)\right)-1} \\
&\lesssim& m\cdot \omega^4,
\end{eqnarray*}
which implies that
$$
\abs{\hat{\phi}_\ell^{(m)}(\omega)}\,\, \lesssim\,\,  \sqrt{m\choose \ell}\cdot \left(\frac{\omega}{4}\right)^\ell\cdot m\cdot \omega^4.
$$

Finally,  the conclusion of (iii)  can be obtained by Lemma \ref{le:minusless1m} directly.

\end{proof}
\begin{lem}\label{le:rm0}
Let  $R_m$ be given by \eqref{eq:rm0}. Then
$$
R_m\lesssim \frac{\ln^5m}{m}; \quad\hbox{and}\quad
\lim_{m\rightarrow \infty}R_m=0.
$$
\end{lem}
\begin{proof}  Let
\begin{eqnarray*}
R_{m,1}&:=&\sup_{1\leq \abs{\omega}\leq 2} \sum_{n\in \Z}\sum_{\ell=1}^m\abs{\hat{\phi}_\ell^{(m)}(2^n\omega)}^2,\\
R_{m,2}&:=&\sup_{1\leq \abs{\omega}\leq 2}\sum_{k\in \Z\setminus\{0\}}\sum_{n\in \Z} \sum_{\ell=1}^m \left(\abs{\hat{\phi}_\ell^{(m)}(2^n\omega)}\cdot \abs{\hat{\phi}_\ell^{(m)}(2^n\omega+2k\pi)}\right).
\end{eqnarray*}
Then, we  have that
\begin{equation}\label{eq:rm}
R_m \leq R_{m,1}+R_{m,2}.
\end{equation}

To estimate $R_m$, we  consider $R_{m,1}$ and $R_{m,2}$, respectively. We first estimate $R_{m,1}$.  For this, we rewrite
\begin{eqnarray*}
R_{m,1}&=&\sup_{1\leq \abs{\omega}\leq 2}\sum_{n\in \Z} \sum_{\ell=1}^m\, \abs{\hat{\phi}_\ell^{(m)}(2^n\omega)}^2
= \sup_{1\leq \abs{\omega}\leq 2} [S_1(\omega)+S_2(\omega)+S_3(\omega)],
\end{eqnarray*}
where
\begin{eqnarray*}
S_1(\omega)&:=& \sum_{n\geq 5} \sum_{\ell=1}^m\,\, \abs{\hat{\phi}_\ell^{(m)}(2^n\omega)}^2,\qquad
S_2(\omega):=\sum_{ -\lfloor\log_2m\rfloor<n<5} \sum_{\ell=1}^m\,\, \abs{\hat{\phi}_\ell^{(m)}(2^n\omega)}^2,\\
S_3(\omega)&:=&\sum_{n\leq -\lfloor\log_2m\rfloor} \sum_{\ell=1}^m\,\, \abs{\hat{\phi}_\ell^{(m)}(2^n\omega)}^2.
\end{eqnarray*}
By  (i) in Lemma \ref{le:upbound},  we obtain that, for $1\leq \abs{\omega}\leq 2$,
\begin{eqnarray*}
S_1(\omega)&=& \sum_{n\geq 5}\sum_{\ell=1}^m \abs{\hat{\phi}_\ell^{(m)}(2^n\omega)}^2\\
& \leq & 4  \sum_{n\geq 5} \sum_{\ell=1}^m  {m\choose \ell} \frac{4^{m+\ell}}{(2^n\omega)^{2m}} \\
& = & 4 \sum_{n\geq 5} \frac{4^m}{(2^n\omega)^{2m}}\sum_{\ell=1}^m  {m\choose \ell} 4^\ell \\
&\lesssim & \left(\frac{5}{256}\right)^m.
\end{eqnarray*}
Using (ii) in Lemma \ref{le:upbound}, when $1\leq \abs{\omega} \leq 2$,
\begin{eqnarray*}
S_3(\omega)&=& \sum_{n\leq -\lfloor\log_2m\rfloor}\sum_{\ell=1}^m\abs{\hat{\phi}_\ell^{(m)}(2^n\omega)}^2= \sum_{n \geq
\lfloor\log_2m\rfloor}\sum_{\ell=1}^m \abs{\hat{\phi}_\ell^{(m)}(\frac{\omega}{2^n})}^2\\
&\lesssim & \sum_{n\geq \lfloor\log_2m\rfloor} \sum_{\ell=1}^m  {m\choose \ell} \left(\frac{\omega}{4\cdot 2^n}\right)^{2\ell} \cdot m^2 \cdot \left(\frac{\omega^4}{2^{4n}}\right)^2 \\
&=& \sum_{n\geq \lfloor\log_2m\rfloor}m^2 \cdot \left(\frac{\omega^4}{2^{4n}}\right)^2  \sum_{\ell=1}^m {m\choose \ell} \left(\frac{\omega}{4\cdot 2^n}\right)^{2\ell} \\
&\leq & m^2 \sum_{n\geq \lfloor\log_2m\rfloor } \left(\frac{\omega^4}{2^{4n}}\right)^2\left(1+\left(\frac{\omega}{4\cdot 2^n}\right)^2\right)^m \\
&\lesssim &  m^2 \sum_{n\geq \lfloor\log_2m\rfloor} \left(\frac{1}{2^{4n}}\right)^2\left(1+\left(\frac{1}{ 2^{n+1}}\right)^2\right)^m\\
&\lesssim & \frac{1}{m^6}.
\end{eqnarray*}
Here, the last inequality uses the fact of  $\{\left(1+\left(\frac{1}{ 2^{n+1}}\right)^2\right)^m\}_{n\geq \lfloor\log_2m\rfloor,\, m\in \Z_+}$ being a bounded sequence and
\begin{eqnarray*}
\sum_{n \geq \lfloor\log_2m\rfloor} \left(\frac{1}{2^{4n}}\right)^2 =\sum_{n\geq \lfloor\log_2m\rfloor} \frac{1}{256^n} = \frac{256}{255} \cdot \frac{1}{256^{\lfloor\log_2m\rfloor}} \lesssim \frac{1}{m^8}.
\end{eqnarray*}
Moreover, by (iii) in Lemma \ref{le:upbound},  we have
$$
S_3(\omega)=\sum_{\ell=1}^m  \sum_{-\lfloor\log_2m\rfloor \leq n\leq 5}\abs{\hat{\phi}_\ell(2^n\omega)}^2 \lesssim  \frac{\ln^5m}{m}.
$$
Combining the results above, we obtain that
\begin{equation}\label{eq:rm1}
R_{m,1}=\sup_{1\leq \abs{\omega}\leq 2} [S_1(\omega)+S_2(\omega)+S_3(\omega)] \lesssim  \frac{\ln^5m}{m}.
\end{equation}
We next turn to
$$
R_{m,2}=\sup_{1\leq \abs{\omega}\leq 2}\sum_{k\in \Z\setminus\{0\}} \sum_{\ell=1}^m\sum_{n\in \Z} \left(\abs{\hat{\phi}_\ell^{(m)}(2^n\omega)}\cdot \abs{\hat{\phi}_\ell^{(m)}(2^n\omega+2k\pi)}\right).
$$
To state conveniently, we set
$$
\beta(2k\pi):=\sup_{1\leq \abs{\omega}\leq 2}\sum_{\ell=1}^m  \sum_{n\in \Z}\, \abs{\hat{\phi}_\ell^{(m)}(2^n\omega)}\cdot \abs{\hat{\phi}_\ell^{(m)}(2^n\omega+2k\pi)}.
$$
Then
$$
R_{m,2}\leq \sum_{k\in \Z\setminus\{0\}} \beta(2k\pi).
$$
Set $k_0:=10$. When $1\leq \abs{k}\leq k_0-1$, using the argument similar to the one in the estimation of $R_{m,1}=\beta(0)$, we can show that
$$
\beta(2k\pi) \lesssim \frac{1}{m}\qquad \mbox{ for } 1\leq \abs{k}\leq k_0-1.
$$
We claim that, when  $\abs{k}\geq k_0$,
\begin{equation}\label{eq:claim}
\beta(2k\pi)\lesssim \frac{3^m}{{\abs{2k\pi}}^{m/2}}.
\end{equation}
And hence,
\begin{equation}\label{eq:rm2}
R_{m,2}=\left(\sum_{1\leq \abs{k}\leq k_0-1}+\sum_{\abs{k}\geq k_0}\right)\beta(2k\pi)\lesssim \frac{1}{m}+\sum_{\abs{k}\geq k_0}\frac{3^m}{{\abs{2k\pi}}^{m/2}}\lesssim \frac{1}{m}.
\end{equation}
Combing \eqref{eq:rm1} and \eqref{eq:rm2}, we obtain that
$$
R_m=R_{m,1}+R_{m,2}\lesssim \frac{\ln^5m}{m},
$$
 which implies the conclusion.

Finally,  we prove  \eqref{eq:claim}. A simple observation is that $\beta(2k\pi)=\beta(-2k\pi)$. Hence, we only
need consider the case where $k\geq k_0$.  For the  convenience, let
\begin{eqnarray*}
\beta_+(2k\pi)&:=&\sup_{1\leq \abs{\omega}\leq 2}\sum_{\ell=1}^m  \sum_{n\in \Z_+}  \left(\abs{\hat{\phi}_\ell^{(m)}(2^n\omega)}\cdot \abs{\hat{\phi}_\ell^{(m)}(2^n\omega+2k\pi)}\right),\\
\beta_-(2k\pi)&:=&\sup_{1\leq \abs{\omega}\leq 2} \sum_{\ell=1}^m  \sum_{n<0} \left(\abs{\hat{\phi}_\ell^{(m)}(2^n\omega)}\cdot \abs{\hat{\phi}_\ell^{(m)}(2^n\omega+2k\pi)}\right).
\end{eqnarray*}
Then $\beta(2k\pi)\leq \beta_+(2k\pi)+\beta_-(2k\pi)$.
To estimate $\beta_+(2k\pi)$, we furthermore set
\begin{eqnarray*}
\beta_+^+(2k\pi)&:=& \sup_{1\leq {\omega} \leq 2} \sum_{\ell=1}^m \sum_{n\in \Z_+}\abs{\hat{\phi}_\ell^{(m)}(2^n\omega)}\cdot \abs{\hat{\phi}_\ell^{(m)}(2^n\omega+2k\pi)}, \\
\beta_+^-(2k\pi)&:=&\sup_{-2\leq {\omega} \leq -1} \sum_{\ell=1}^m \sum_{n\in \Z_+}\abs{\hat{\phi}_\ell^{(m)}(2^n\omega)}\cdot \abs{\hat{\phi}_\ell^{(m)}(2^n\omega+2k\pi)}.
\end{eqnarray*}
Then $\beta_+(2k\pi)=\max\{\beta_+^+(2k\pi),\beta_+^-(2k\pi)\}$.
Note that  $k\geq k_0=10$,  by Lemma \ref{le:upbound},
\begin{eqnarray*}
\beta_+^+(2k\pi)&=&\sup_{1\leq \omega\leq 2} \sum_{\ell=1}^m \sum_{n\in \Z_+}\abs{\hat{\phi}_\ell^{(m)}(2^n\omega)}\cdot \abs{\hat{\phi}_\ell^{(m)}(2^n\omega+2k\pi)}\\
&\lesssim& \sum_{\ell=1}^m \sum_{n\in \Z_+} {m\choose \ell} \frac{2^{m+\ell}}{(2^n+2k\pi)^m} \\
&=& \sum_{n\in\Z_+}\sum_{\ell=1}^m {m\choose \ell} \frac{2^{m+\ell}}{(2^n+2k\pi)^m} \\
&\leq & \sum_{n\in \Z_+}\frac{6^m}{(2^n+2k\pi)^m} \leq \sum_{n\in \Z_+} \frac{6^m}{2^m\cdot 2^{nm/2}\cdot (2k\pi)^{m/2}}\\
&\lesssim & \frac{3^m}{({2k\pi})^{m/2}}.
\end{eqnarray*}
We next consider
\begin{eqnarray*}
\beta_+^-(2k\pi)&=& \sup_{-2\leq \omega\leq -1} \sum_{\ell=1}^m \sum_{n\in \Z_+}\abs{\hat{\phi}_\ell^{(m)}(2^n\omega)}\cdot \abs{\hat{\phi}_\ell^{(m)}(2^n\omega+2k\pi)}\\
&=&\sup_{1\leq \omega\leq 2} \sum_{\ell=1}^m  \sum_{n\in \Z_+}\abs{\hat{\phi}_\ell^{(m)}(2^n\omega)}\cdot \abs{\hat{\phi}_\ell^{(m)}(2^n\omega-2k\pi)}.
\end{eqnarray*}
 A simple observation is that $\max\{\abs{2^n\omega},\abs{2^n\omega-2k\pi}\}\geq k\pi \geq k_0\pi$.  Then using (i) in Lemma
\ref{le:upbound}, we obtain that
$$
\abs{\hat{\phi}_\ell^{(m)}(2^n\omega)}\cdot \abs{\hat{\phi}_\ell^{(m)}(2^n\omega-2k\pi)}\lesssim {m\choose \ell}\cdot \frac{2^{m+\ell}}{(k\pi)^m}.
$$
Set $n_0:=\lceil\log_2(2k\pi)\rceil+5$. Then,
\begin{eqnarray*}
& &\sup_{1\leq \omega \leq 2} \sum_{\ell=1}^m \sum_{0\leq n\leq n_0}\abs{\hat{\phi}_\ell^{(m)}(2^n\omega)}\cdot \abs{\hat{\phi}_\ell^{(m)}(2^n\omega-2k\pi)}\\
&\lesssim & \sum_{0\leq n\leq n_0} \sum_{\ell=1}^m {m\choose \ell} \frac{2^{m+\ell}}{(k\pi)^m} =  \sum_{0\leq n\leq n_0}
  \frac{2^{m}}{(k\pi)^m} \sum_{\ell=1}^m {m\choose \ell}\cdot 2^\ell \\
  &\leq &\sum_{0\leq n\leq n_0} \frac{6^m}{(k\pi)^m} \lesssim \log_2(2k\pi)\cdot \frac{6^m}{(k\pi)^m}.
\end{eqnarray*}
We next consider
$$
\sup_{1\leq \omega \leq 2}\sum_{\ell=1}^m \sum_{n_0+1\leq n}\abs{\hat{\phi}_\ell^{(m)}(2^n\omega)}\cdot \abs{\hat{\phi}_\ell^{(m)}(2^n\omega-2k\pi)}.
$$
By (i) in Lemma \ref{le:upbound}, when $n\geq n_0+1$,
$$
\sup_{1\leq \omega \leq 2}\abs{\hat{\phi}_\ell^{(m)}(2^n\omega)}\leq \sqrt{m\choose \ell} \frac{2^{m+\ell}}{(2k\pi)^m}, \quad \sup_{1\leq \omega \leq 2}\abs{\hat{\phi}_\ell^{(m)}(2^n\omega-2k\pi)}\leq \sqrt{m\choose \ell}\frac{2^{m+\ell}}{(2^n-2k\pi)^m}.
$$
Hence,
\begin{eqnarray*}
& & \sup_{1\leq \omega \leq 2}\sum_{\ell=1}^m\sum_{n\geq n_0+1}\abs{\hat{\phi}_\ell^{(m)}(2^n\omega)}\cdot \abs{\hat{\phi}_\ell^{(m)}(2^n\omega-2k\pi)} \\
&\leq & \sum_{\ell=1}^m \sum_{n\geq n_0+1} {m\choose \ell} \frac{2^{m+\ell}}{(2k\pi)^m}\cdot \frac{2^{m+\ell}}{(2^n-2k\pi)^m}\\
& \leq & \frac{4^m\cdot 5^m}{(2k\pi)^m} \sum_{n\geq n_0+1} \frac{1}{(2^n-2k\pi)^m}\lesssim \left(\frac{10}{k\pi}\right)^m,\qquad k\geq k_0.
\end{eqnarray*}
Therefore,
$$
\beta_+^-(2k\pi)\lesssim \log_2(2k\pi)\cdot \frac{6^m}{(k\pi)^m}+\left(\frac{10}{k\pi}\right)^m.
$$
This concludes that
$$
\beta_+(2k\pi)=\max\{\beta_+^+(2k\pi),\beta_+^-(2k\pi)\} \lesssim \frac{3^m}{({2k\pi})^{m/2}}.
$$
Using (ii) in Lemma \ref{le:upbound} and a similar analysis with above, we can obtain that
\begin{eqnarray*}
\beta_-(2k\pi)&=&  \sup_{1\leq \abs{\omega}\leq 2} \sum_{\ell=1}^m \sum_{n<0}\,\abs{\hat{\phi}_\ell^{(m)}(2^n\omega)}\cdot \abs{\hat{\phi}_\ell^{(m)}(2^n\omega+2k\pi)}\\
&\lesssim & m\cdot \frac{4^m}{(2k\pi)^m}.
\end{eqnarray*}
Putting everything together, we have that
$$
\beta(2k\pi)=\beta_+(2k\pi)+\beta_-(2k\pi) \lesssim \frac{6^m}{(2\sqrt{2k\pi})^m}+ m\cdot \frac{4^m}{(2k\pi)^m}\lesssim \frac{3^m}{({2k\pi})^{m/2}}.
$$
This proves \eqref{eq:claim}.

\end{proof}

\begin{proof}[{\bf Proof of Theorem \ref{th:gaussframe}}]
Recall that
$$
\Phi^{(m)}=\{\phi^{(m)}_1,\ldots,\phi^{(m)}_m\},\,\, \phi_\ell^{(m)}=\psi_\ell^{(m)}-{G}_\ell^{(m)}, \quad \ell=1,\ldots,m;
$$
 and that $X(\Psi^{(m)})$ is a tight frame with frame bound 1, where $\Psi^{(m)}=\{\psi_1^{(m)},\ldots,\psi_m^{(m)}\}$.
 Lemma \ref{le:rm0} shows that  $X(\Phi^{(m)})$ is a Bessel sequence with a bound $R_m\rightarrow 0$.
  Then Theorem \ref{th:FZ} leads that $X(G^{(m)})$ is a frame with
frame bound $A_m=(1-\sqrt{R_m})^2$ and $B_m=(1+\sqrt{R_m})^2$ as $m$ sufficiently large. Furthermore, it can be close to a tight frame, since $\lim_{m\rightarrow \infty}A_m=\lim_{m\rightarrow \infty}B_m=1$, which completes the proof.
\end{proof}

\begin{remark}
Although Theorem \ref{th:gaussframe} confirms the case where $m$ is sufficiently large, the result of Theorem \ref{th:gaussframe} seems to hold for small $m$ ($m$ can be as small as $2$). For small $m$, combining \eqref{eq:rm0} and Theorem \ref{th:FZ}, we can estimate the frame bounds of $X(G^{(m)})$ numerically. We list the frame bound estimation of $X(G^{(m)})$, $2\leq m\leq 8$,
in  Table 1, which clearly shows the frame property of $X(G^{(m)})$ for small $m$. For example,
for $m=2$, $X(G^{(2)})$ is a frame with frame bounds $A \approx  0.3855$ and    $B\approx 1.9020$.
 Figure 2 shows that the graphs of the functions of $G_1^{(2)}$ and $G_2^{(2)}$, respectively.

\begin{table}
\caption{ The numerical results of frame bounds of $X(G^{(m)})$}
\begin{tabular}{c c c c c c c c}
  \hline
  $m$ & 2 & 3 & 4 & 5 & 6 & 7 & 8  \\ \hline
  $A$ & 0.3855 & 0.5266 &  0.5898 &  0.6407  & 0.6803 &  0.7095  &  0.7274  \\
  $B$ & 1.9020 & 1.6239 & 1.5179 & 1.4390 & 1.3811 & 1.3403 & 1.3159  \\
  \hline
\end{tabular}
\end{table}

\begin{figure}[!ht]
\begin{center}
\epsfxsize=5.5cm\epsfbox{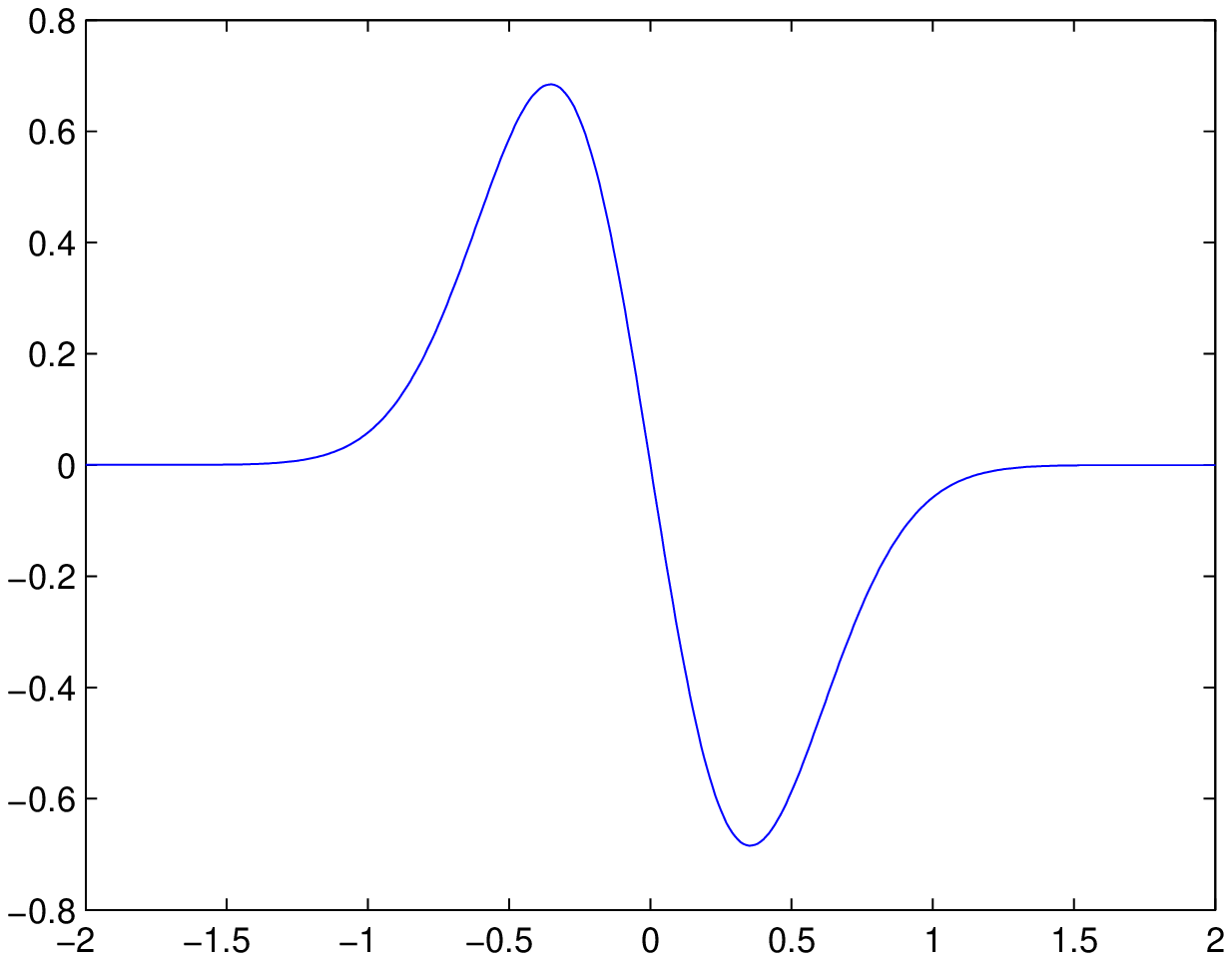}\epsfxsize=5.5cm\epsfbox{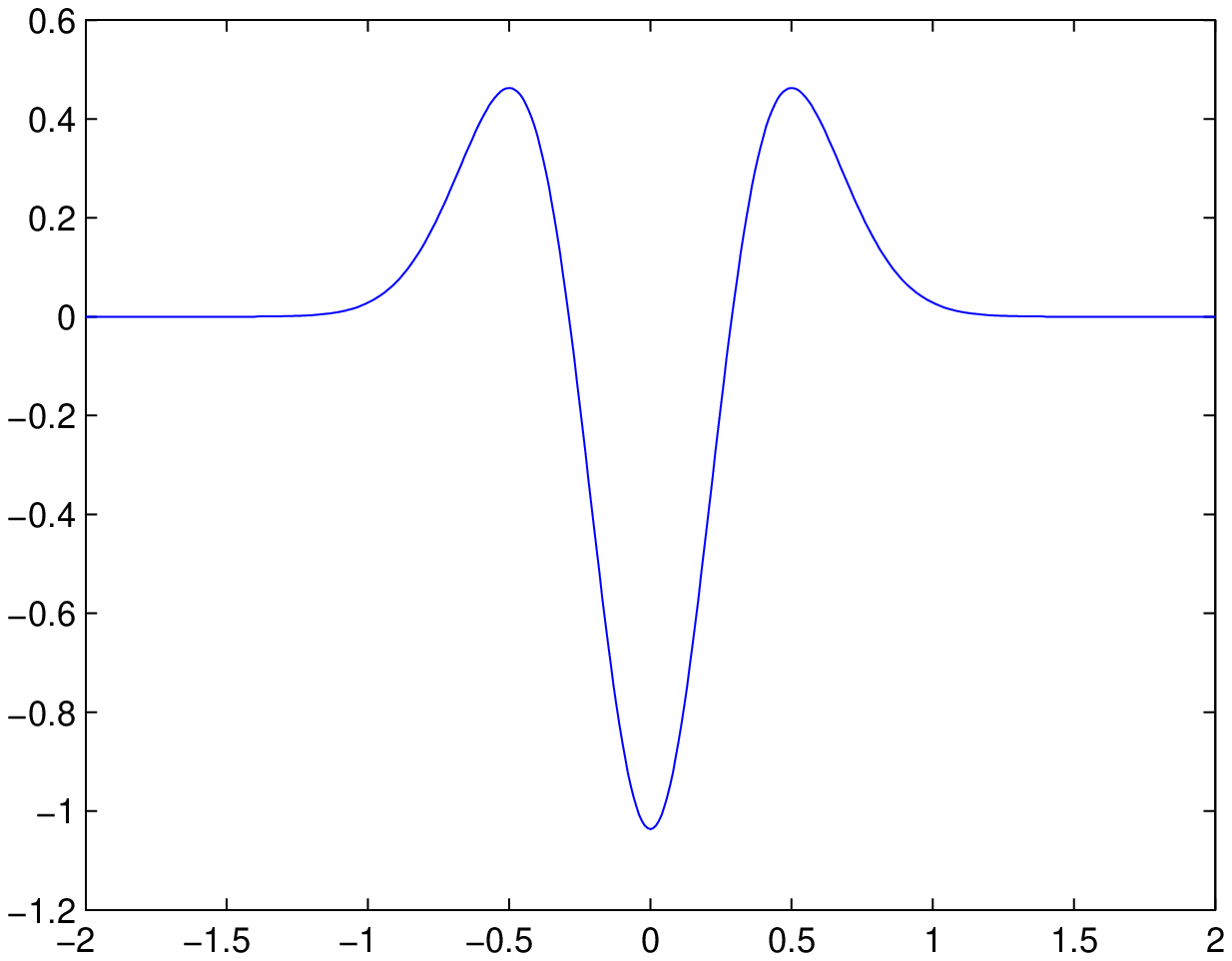}
\end{center}
\caption{ The graphs of $G^{(2)}_1(x)=-\sqrt{{32}/{\pi}} x \exp(-4x^2)$ (left) and
$G^{(2)}_2(x)=\sqrt{{27}/{(8\pi)}}(12x^2-1)\exp(-6x^2)$ (right).}
\end{figure}

\end{remark}

\end{document}